\documentclass[reqno]{amsart}
\usepackage{amsmath, amssymb, amsthm, epsfig}
\usepackage{hyperref, latexsym}
\usepackage{url}
\usepackage[mathscr]{euscript}

\usepackage{color}
\usepackage{fullpage} 
\usepackage{setspace}

\usepackage{array,booktabs}

\def\today{\ifcase\month\or
  January\or February\or March\or April\or May\or June\or 
  July\or August\or September\or October\or November\or December\fi
  \space\number\day, \number\year}

 \newtheorem{theorem}{Theorem}
  
 \newtheorem{lemma}[theorem]{Lemma}

 \newcommand{\mc}{\mathcal}

 \newcommand{\M}{\mc{M}}

 \newcommand{\R}{\mathbb{R}}
  
 \newcommand{\N}{\mathbb{N}}
 
 \newcommand{\Z}{\mathbb{Z}}

 \newcommand{\dt}{\text{\rm d}t}

 \newcommand{\dx}{\text{\rm d}x}
 
 \newcommand{\dy}{\text{\rm d}y}

\newcommand{\wt}{\widetilde}

\newcommand{\var}{{\rm Var\,}}

\begin{document}

\title[]{Endpoint Sobolev and BV continuity \\
for maximal operators}
\author[Carneiro, Madrid and Pierce]{Emanuel Carneiro, Jos\'e Madrid and Lillian B. Pierce}
\date{\today}

\address{IMPA - Instituto Nacional de Matem\'{a}tica Pura e Aplicada, Estrada Dona Castorina 110, Rio de Janeiro - RJ, Brazil, 22460-320.}
\email{carneiro@impa.br}

\address{Department of Mathematics, P.O. Box 11100, FI-00076 Aalto University, Finland.}

\email{jose.madridpadilla@aalto.fi}

\address{Department of Mathematics, Duke University, 120 Science Drive, Durham - NC, 27708, USA.}
\email{pierce@math.duke.edu}

\allowdisplaybreaks
\numberwithin{equation}{section}

\maketitle

\begin{abstract}
In this paper we investigate some questions related to the continuity of maximal operators in $W^{1,1}$ and $BV$ spaces, complementing some well-known boundedness results. Letting $\widetilde M$ be the one-dimensional uncentered Hardy-Littlewood maximal operator, we prove that the map $f \mapsto \big(\widetilde Mf\big)'$ is continuous from $W^{1,1}(\R)$ to $L^1(\R)$. In the discrete setting, we prove that $\widetilde M: BV(\Z) \to BV(\Z)$ is also continuous. For the one-dimensional fractional Hardy-Littlewood maximal operator, we prove by means of counterexamples that the corresponding continuity statements do not hold, both in the continuous and discrete settings, and for the centered and uncentered versions.
\end{abstract}

\section{Introduction} 

The purpose of this paper is to present new results and discuss some open problems related to the continuity of the one-dimensional Hardy-Littlewood maximal operator in $W^{1,1}$ and $BV$ spaces. Such continuity questions for maximal operators in Sobolev spaces are already nontrivial in the case $W^{1,p}$ with $p>1$  (see \cite{Lu1, Lu2}) and, to the best of our knowledge, this is the first time that the more subtle limiting case $p=1$ is addressed in the literature. We start with a brief recollection of some of the recent developments on the general regularity theory for maximal operators.

\subsection{Background} For $f \in L^1_{loc}(\R^d)$ we define the centered Hardy-Littlewood maximal function $Mf$ by 
\begin{equation}\label{Intro_max}
Mf(x) = \sup_{r > 0} \frac{1}{m(B_r(x))} \int_{B_r(x)} |f(y)|\,\dy\,,
\end{equation}
where $B_r(x)$ is the open ball of center $x$ and radius $r$, and $m(B_r(x))$ denotes its $d$-dimensional Lebesgue measure. The uncentered maximal function $\wt{M}f$ at a point $x$ is defined analogously, taking the supremum of averages over open balls that contain the point $x$, but that are not necessarily centered at $x$. One of the pillars of harmonic analysis is the Hardy-Littlewood-Wiener theorem, which states that $M: L^p(\R^d) \to L^p(\R^d)$ is a bounded operator when $1 < p \leq \infty$. The classical consequence of an $L^p$-bound for a maximal operator is the pointwise convergence (a.e.) of a particular sequence of associated objects (in this case, the integral averages as the the radius goes to zero). Since $M$ is sublinear, it follows directly from the boundedness statement that $M: L^p(\R^d) \to L^p(\R^d)$ is also a continuous map when $1 < p \leq \infty$. Similar results hold for $\wt{M}$. 

\smallskip
  
An active topic of current research is the investigation of the regularity properties of maximal operators. One of the driving questions in this theory is whether a given maximal operator improves, preserves or destroys the {\it a priori} regularity of an initial datum $f$. With respect to Sobolev regularity, Kinnunen in the seminal paper \cite{Ki} established that $M: W^{1,p}(\R^d) \to W^{1,p}(\R^d)$ is a bounded operator for $1 < p \leq \infty$ (the same holds for $\wt{M}$). Since $M$ is not necessarily sublinear at the derivative level, the continuity of $M: W^{1,p}(\R^d) \to W^{1,p}(\R^d)$, for $p>1$, is a highly nontrivial problem, originally attributed to T. Iwaniec \cite[Question 3]{HO}. This question was settled, affirmatively, by Luiro in \cite{Lu1}, and this is perhaps the paper that most closely inspires the present work. Kinnunen's original result was later extended to a local setting in \cite{KL}, to a fractional setting in \cite{KiSa} and to a multilinear setting in \cite{CM}. Other works on the regularity of maximal operators in Sobolev spaces and other function spaces include \cite{ACP, Ko, Liu, M, Ra, St}.

\smallskip

The understanding of the action of $M$ on the endpoint space $W^{1,1}(\R^d)$ is a much more delicate issue. The question of whether the map $f \mapsto \nabla Mf$ is bounded from $W^{1,1}(\R^d)$ to $L^1(\R^d)$ was raised by Haj\l asz and Onninen in \cite{HO} and remains unsolved in its full generality. A complete solution was reached only in dimension $d=1$ in \cite{AP, Ku, Ta} and partial progress on the general case $d>1$ was obtained by Haj\l asz and Mal\'{y} \cite{HM} and, more recently, by Saari \cite{Saa} and Luiro \cite{Lu3}.

\smallskip

Let us elaborate on the achievements in dimension $d=1$ since these will be quite important for our purposes here. The first works considered the uncentered maximal operator, which is a more regular object than the centered one. If $f \in W^{1,1}(\R)$ then $\wt{M}f$ is weakly differentiable and we have the bound 
\begin{equation}\label{Intro_Tanaka}
\big\|\big(\wt{M}f\big)'\big\|_{L^1(\R)} \leq C \,\|f'\|_{L^1(\R)}.
\end{equation}
This was first established by Tanaka \cite{Ta} with constant $C = 2$ and then refined by Aldaz and P\'{e}rez L\'{a}zaro \cite{AP} who obtained the sharp constant $C=1$ (see also \cite{LCW}). In fact, Aldaz and P\'{e}rez L\'{a}zaro \cite{AP} went further by showing that if $f \in BV(\R)$ then $\wt{M}f$ is absolutely continuous and 
\begin{equation}\label{Intro_1_var_UHL}
\var(\wt{M}f) \leq \var(f)\,,
\end{equation}
where $\var(f)$ denotes the total variation of $f$. Inequalities \eqref{Intro_Tanaka} and \eqref{Intro_1_var_UHL} were proved for the centered maximal operator by Kurka in \cite{Ku}, with constant $C = 240,004$. It is currently unknown if one can bring down the value of such constant to $C = 1$ in the centered case. The variation-diminishing property was established in \cite{CFS, CS} for other maximal operators of convolution type associated to smooth kernels (e.g. the Gauss kernel and the Poisson kernel).

\smallskip

Our aim at present is to study the continuity properties of these maximal operators on the spaces $W^{1,1}(\R)$ and $BV(\R)$, as well as certain discrete and fractional analogues. Depending on the precise setting, we obtain either positive results (proving continuity) or negative results (via counterexamples). The notation and terminology used in this paper are classical in the theory, and in particular align with our previous works \cite{BCHP, CMa, CS} on the subject.

\subsection{${\bf W^{1,1}(\R)-}$continuity}

\smallskip

Inequality \eqref{Intro_Tanaka} establishes the boundedness of the map $f \mapsto \big(\wt{M}f\big)'$ from $W^{1,1}(\R)$ to $L^1(\R)$. Our first result establishes the continuity of this map; this is the main result of the paper.

\begin{theorem}\label{Thm1}
The map $f \mapsto \big(\wt{M}f\big)'$ is continuous from $W^{1,1}(\R)$ to $L^1(\R)$.
\end{theorem}

As we shall see from the proof, the continuity of this endpoint map is quite subtle. In previous literature, there are only two works, both due to H. Luiro \cite{Lu1,Lu2}, that touch right at the heart of what we aim to accomplish in Theorem \ref{Thm1}. In the remarkable paper \cite{Lu1}, Luiro establishes the continuity of $M$ in the $W^{1,p}(\R^d)-$setting for $p>1$. In \cite{Lu2}, a similar Sobolev continuity statement for the local version of the Hardy-Littlewood maximal operator is proved. Yet in both of these papers, Luiro's beautiful argument relies heavily on the fact that $M: L^p(\R^d) \to L^p(\R^d)$ is bounded, and hence it cannot be adapted to the case considered in Theorem \ref{Thm1}. Without this strong tool (the $L^p$-boundedness) in our favor, we must dive deeper into the very nature of these maximal functions to understand how they behave at the derivative level.

\smallskip

In our proof we make use of important ideas and qualitative lemmas developed in some of the innovative works on this theory over the last two decades. The particular works of Haj\l asz and Mal\'{y} \cite{HM}, Luiro \cite{Lu1} and Tanaka \cite{Ta} are quite relevant in our proof. That being said, we stress the fact that the present literature only carries us to a certain point, from which we must diverge to our own original ideas to estimate the crucial oscillation bounds that appear in connection to our problem. Such oscillation bounds are of independent interest and provide genuine new insights to the theory, as well as the broad outline of our method. We use in a decisive way the qualitative description of the uncentered maximal function (and the one-sided maximal functions) on the disconnecting set (see Section \ref{sec_connect_disconnect}).

\smallskip

The analogous result of our Theorem \ref{Thm1} for the centered case is an interesting and natural question, which lies out of reach of our current methods.

\smallskip

\noindent{\bf Question A.} \emph{Is the map $f \mapsto \big(Mf\big)'$ continuous from $W^{1,1}(\R)$ to $L^1(\R)$?}

\subsection{$BV(\R)-$continuity}
The space $BV(\R)$ of functions $f:\R \to \R$ of bounded total variation is a Banach space with the norm
\begin{equation}\label{BV norm}
\|f\|_{{\rm BV(\R)}} = \big|f(-\infty)\big| + \var(f),
\end{equation}
where $f(-\infty):= \lim_{x \to -\infty} f(x)$. 
Since 
$$\big\|\wt{M}f\big\|_{L^{\infty}(\R)}\leq \|f\|_{L^{\infty}(\R)} \leq  \|f\|_{{\rm BV(\R)}},$$
in combination with inequality \eqref{Intro_1_var_UHL}, this tells us that $\wt{M}: BV(\R) \to BV(\R)$ is a bounded operator (the same holds for $M$). The corresponding $BV-$continuity statements arise as interesting problems that are qualitatively stronger that Theorem \ref{Thm1} and Question A, if confirmed.

\smallskip

\noindent{\bf Question B.} \emph{ Is the map $\wt{M} : BV(\R) \to BV(\R)$ continuous?}

\smallskip

\noindent{\bf Question C.} \emph{Is the map $M : BV(\R) \to BV(\R)$ continuous?}

\smallskip

\noindent{\bf Remark.} If one considers only the seminorm $\var(\cdot)$ on $BV(\R)$, note that there are sequences $\{f_j\}_{j\geq 1} \subset BV(\R)$ and $f \in BV(\R)$ such that $\var(f_j - f) \to 0$ as $j \to \infty$ but $\var(\wt{M}f_j - \wt{M}f) \nrightarrow 0$. Take for instance $f(x) = (1 - |x|) \chi_{[-1,1]}(x)$ and $f_j(x) = f(x) - 1$ for all $j \geq 1$. Then $\wt{M}f_j \equiv 1$ while $\wt{M}f$ is not constant. There are a couple of natural ways to avoid such pathological examples, for instance:
\begin{enumerate}
\item[(i)] Assume that $f \geq 0$ and  $f_j \geq 0$ for all $j$ (since the maximal operator acts on $|f|$).
\item[(ii)] Adopt some sort of normalization. In our case, this is represented by the term $|f(-\infty)|$ in the $BV-$norm \eqref{BV norm}.
\end{enumerate}
With this in mind, we can actually show that situation (ii) is more general than (i), which justifies our choice. In fact, assuming a positive answer for Question B, let us consider a sequence $\{f_j\}_{j\geq 1} \subset BV(\R)$ and $f \in BV(\R)$ such that $f_j \geq 0$ and  $f \geq 0$, and $\var(f_j - f) \to 0$ as $j \to \infty$, and let us show that $\var(\wt{M}f_j - \wt{M}f) \to 0$ as well. Let $c = \lim_{x\to -\infty} f(x)$ and $c_j = \lim_{x\to -\infty} f_j(x)$. Define $g := f +1$ and $g_j = f_j + c- c_j +1$. Note that $c+1 = \lim_{x\to -\infty} g(x) = \lim_{x\to -\infty} g_j(x)$. Since we aligned the left limits, a positive answer for Question B would imply that  $\var(\wt{M} g_j - \wt{M} g) \to 0$. Finally, note that there exists $j_0$ such that for $j \geq j_0$ we have $\var(g_j - g) = \var(f_j - f) \leq 1$, and thus since $g(x)\geq 1$ by construction, we may conclude $g_j(x) \geq 0$ for all $x \in \R$ and $j \geq j_0$. In general, for any function $h(x) = \tilde{h}(x) + \tilde{c}$ such that $h(x) \geq 0$ and $\tilde{h}(x) \geq 0$ for all $x \in \R$, we have $\wt{M}h = \wt{M}\tilde{h} + \tilde{c}.$ We may deduce that $\wt{M}g = \wt{M}f + 1$ and  $\wt{M}g_j = \wt{M}f_j + c - c_j + 1$ for $j \geq j_0$, which gives us $\var(\wt{M}f_j - \wt{M}f) \to 0$.

\subsection{Discrete analogues} For a discrete function $f:\Z \to \R$ we define its $\ell^p(\Z)-$norm as usual
\begin{equation*}
\|f\|_{\ell^{p}{( \Z)}}:= \left(\sum_{n\in \Z} {|f(n)|^{p}}\right)^{1/p},
\end{equation*}
if $1\leq p<\infty$, and
\begin{equation*}
\|f\|_{\ell^{\infty}{(\Z)}}:= \sup_{n\in\Z}{|f(n)|}.
\end{equation*}
We define its discrete derivative by $f'(n) := f(n+1) - f(n)$ and its total variation by
\begin{equation*}
\var(f) := \|f'\|_{\ell^1(\Z)} = \sum_{n\in \Z} |f(n+1) - f(n)|.
\end{equation*}
Let $BV(\Z)$ be the space of discrete functions of bounded total variation. This is a Banach space with the norm 
\begin{equation}\label{BV norm discrete}
\|f\|_{{\rm BV(\Z)}} = \big|f(-\infty)\big| + \var(f),
\end{equation}
where $f(-\infty):= \lim_{n\to -\infty} f(n)$. 

\smallskip

For $f:\Z \to \R$ we define the discrete uncentered Hardy-Littlewood maximal function $\wt{\M}f:\Z \to \R^+$ by 
\begin{equation}\label{disc_HLM}
\wt{\M}f(n) = \sup_{\stackrel{r,s \geq 0}{r,s \in \Z}} \frac{1}{(r + s +1)} \sum_{k = -r}^{s} |f(n + k)|.
\end{equation}
The discrete analogue of the sharp inequality \eqref{Intro_1_var_UHL} was established by Bober, Carneiro, Hughes and Pierce in \cite{BCHP}, who showed that
\begin{equation}\label{Intro_BCHP}
\var(\wt{\M}f) \leq \var(f).
\end{equation}
In the centered case, inequality \eqref{Intro_BCHP} with a constant $C >1$ was obtained by Temur in \cite{Te}.

\smallskip

Inequality \eqref{Intro_BCHP}, together with the simple fact that $|\wt{\M}f(-\infty)| \leq |f(-\infty)|+ \var(f)$ (which follows for example from Lemma \ref{lemma_min_max} in Section \ref{sec_prelim_discrete}), establishes the boundedness of $\wt{\M} :BV(\Z) \to BV(\Z)$. Our second result establishes the continuity of this map, answering the discrete analogue of Question B.
 
\begin{theorem}\label{Thm2}
The map $\wt{\M} :BV(\Z) \to BV(\Z)$ is continuous.
\end{theorem}

With the stronger hypothesis that $\{f_j\}_{j=1}^{\infty} \subset \ell^1(\Z)$ and $f_j \to f$ in $\ell^1(\Z)$, the conclusion that $\big(\wt{\M}f_j\big)' \to \big(\wt{\M}f\big)'$ in $\ell^1(\Z)$ holds, as shown in \cite[Theorem 1]{CH} (both in the centered and uncentered cases). In the discrete setting, note that the space $W^{1,1}(\Z)$ is merely $\ell^1(\Z)$ with an equivalent norm. Therefore, we adopt here the convention that $W^{1,1}(\Z)-$continuity is the same as $\ell^1(\Z)-$continuity just described. In this regard, Theorem \ref{Thm2} is the natural extension of \cite[Theorem 1]{CH} for the uncentered case, establishing the qualitatively stronger $BV-$continuity. This naturally leaves the open problem:

\smallskip

\noindent{\bf Question D.} \emph{Let $\M$ be the discrete centered Hardy-Littlewood maximal operator. Is the map $\M :BV(\Z) \to BV(\Z)$ continuous?}

\smallskip

\noindent{\bf Remark.} Our proof of Theorem \ref{Thm2} relies on the fact that any strict local maximum of $\wt{\M} f$ must be a point of contact with $|f|$, that is, a point $n$ at which $\widetilde{\M}f(n) = |f|(n)$. An even stronger statement holds for the one-dimensional discrete heat flow maximal operator and the one-dimensional discrete Poisson maximal operator defined in \cite{CS}: namely, that these are actually convex on each detachment interval (that is, an interval between two distinct consecutive points of contact, in which necessarily $\widetilde{\M}f(n) > |f|(n)$; see also the disconnecting set we define later in (\ref{detachment})). Our proof of Theorem \ref{Thm2} directly applies to establish the $BV-$continuity of these discrete maximal operators as well. 

\subsection{Real-variable fractional operators} Returning to the setting of $\R^d$, for $0 \leq \beta < d$, we define the centered fractional maximal operator as 
\begin{equation*}
M_{\beta}f(x) = \sup_{r >0} \frac{1}{m(B_r(x))^{1 - \frac{\beta}{d}}} \int_{B_r(x)} |f (y)|\,\dy.
\end{equation*}
When $\beta =0$ we plainly recover \eqref{Intro_max}. The uncentered fractional maximal operator $\wt{M}_{\beta}$ is defined analogously, taking the supremum over uncentered balls. Such fractional maximal operators have applications in potential theory and partial differential equations. By comparison with an appropriate Riesz potential, one can show that if $1 < p < \infty$, $0 < \beta < d/p$ and $q = dp/(d-\beta p)$, then $M_{\beta}: L^p(\R^d) \to L^q(\R^d)$ is bounded. When $p=1$ we have a weak-type bound (for details, see \cite[Chapter V, Theorem 1]{S}). In  \cite[Theorem 2.1]{KiSa}, Kinnunen and Saksman proved that $M_{\beta}: W^{1,p}(\R^d) \to W^{1,q}(\R^d)$ is bounded for $p,q,\beta, d$ as described above, extending Kinunnen's original result \cite{Ki} for the case $\beta =0$. 

\smallskip

The regularity of these operators at the endpoint space $W^{1,1}$ was only recently investigated by Carneiro and Madrid in \cite{CMa}, in the case of the one-dimensional uncentered fractional maximal operator $\wt{M}_{\beta}$. In order to state the results of \cite{CMa} let us define, for a function $f: \R \to \R$ and $1\leq q < \infty$, its {\it Riesz $q$-variation} by 
\begin{equation}\label{Intro_q_var}
\var_q(f) := \sup_{\mc{P}} \left(\sum_{n=1}^{N-1} \frac{|f(x_{n+1}) - f(x_n)|^q}{|x_{n+1} - x_n|^{q-1} }\right)^{1/q},
\end{equation}
where the supremum is taken over all finite partitions $\mc{P} = \{x_1 < x_2 < \ldots < x_N\}$ (see, for instance, the discussion in \cite{BaLi} for this object and its generalizations). A classical result of F. Riesz (see \cite[Chapter IX \S4, Theorem 7]{N}) establishes, in the case $1 < q < \infty$, that $\var_q(f) < \infty$ if and only if $f$ is absolutely continuous and its derivative $f'$ belongs to $L^q(\R)$. Moreover, in this case, one has $\|f'\|_{L^q(\R)} = \var_q(f)$.

\smallskip

We now focus on the case of $\R$. Let $0 < \beta < 1$ and $q = 1/(1-\beta)$. Carneiro and Madrid in \cite[Theorem 1]{CMa} show the following: if $f \in BV(\R)$ is such that $\wt{M}_{\beta}f \not\equiv \infty$, then $\wt{M}_{\beta}f$ is absolutely continuous and its derivative satisfies
\begin{equation}\label{Intro_q_bound1}
\big\|\big(\wt{M}_{\beta}f\big)'\big\|_{L^q(\R)} = \var_q\big(\wt{M}_{\beta}f\big) \leq 8^{1/q}\,\var(f).
\end{equation}
Hence the map $f \mapsto \big(\wt{M}_{\beta}f\big)'$ is bounded from $BV(\R)$ to $L^q(\R)$ \footnote{With the understanding that if $\wt{M}_{\beta}f \equiv \infty$ then $\big(\wt{M}_{\beta}f\big)' \equiv 0$.}
 (note that this implies it is also bounded from $W^{1,1}(\R)$ to $L^q(\R)$). Our next result takes us in a different direction than Theorems \ref{Thm1} and \ref{Thm2}, showing that this map is not continuous.
 
\begin{theorem}\label{Thm3}
Let $0 < \beta < 1$ and $q = 1/(1-\beta)$. The map $f \mapsto \big(\wt{M}_{\beta}f\big)'$ is not continuous from $BV(\R)$ to $L^q(\R)$, i.e. there is a sequence $\{f_j\}_{j\geq 1} \subset BV(\R)$ and a function $f \in BV(\R)$ such that $\|f_j - f\|_{BV(\R)} \to 0$ as $j \to \infty$ but 
$$\big\|\big(\wt{M}_{\beta}f_j\big)' - \big(\wt{M}_{\beta}f\big)'\big\|_{L^q(\R)} = \var_q \big(\wt{M}_{\beta}f_j - \wt{M}_{\beta}f\big) \nrightarrow 0$$ 
as $j \to \infty$.
\end{theorem}

It is interesting to investigate the situation when one considers a stronger form of convergence on the source space. 

\smallskip

\noindent{\bf Question E.} Let $0 < \beta < 1$ and $q = 1/(1-\beta)$. Is the map $f \mapsto \big(\wt{M}_{\beta}f\big)'$ continuous from $W^{1,1}(\R)$ to $L^q(\R)$?

\smallskip

In the centered case, the boundedness of the map $f \mapsto \big(M_{\beta}f\big)'$ from $BV(\R)$ to $L^q(\R)$ (or even from $W^{1,1}(\R)$ to $L^q(\R)$) is still unknown. Independent of the boundedness result, our counterexamples also work to yield the following negative continuity result.

\begin{theorem}\label{Thm4}
Let $0 < \beta < 1$ and $q = 1/(1-\beta)$. There is a sequence $\{f_j\}_{j\geq 1} \subset BV(\R)$ and a function $f \in BV(\R)$ such that $\|f_j - f\|_{BV(\R)} \to 0$ as $j \to \infty$ but $\var_q \big(M_{\beta}f_j - M_{\beta}f\big) \nrightarrow 0$ as $j \to \infty$.
\end{theorem}

At this moment we are not able to produce counterexamples assuming convergence in the more regular space $W^{1,1}(\R)$, and we leave it as the next open problem.

\smallskip

\noindent{\bf Question F.} Let $0 < \beta < 1$ and $q = 1/(1-\beta)$. Is the map $f \mapsto \big(M_{\beta}f\big)'$ bounded and continuous from $W^{1,1}(\R)$ to $L^q(\R)$?

\subsection{Discrete fractional operators}
The endpoint regularity results of \cite{CMa} include also the discrete fractional analogues. In this setting, for $1\leq q < \infty$, we define the $q$-variation of a discrete function $f :\Z\to \R$ by 
\begin{equation*}
\var_q(f) := \left(\sum_{n=-\infty}^{\infty} |f(n+1) - f(n)|^q\right)^{1/q} = \|f'\|_{\ell^q(\Z)}.
\end{equation*}
For $0 \leq \beta < 1$ and $f :\Z\to \R$, we define the one-dimensional discrete uncentered fractional maximal operator by
\begin{equation}\label{Intro_Def_Disc_Max_oper_dim1}
\wt{\M}_{\beta}f(n) = \sup_{\stackrel{r,s \geq 0}{r,s \in \Z}} \,\frac{1}{(r + s+1)^{1 - \beta}} \sum_{k = -r}^{s} |f(n +k)|.
\end{equation}
If $\beta =0$ we recover \eqref{disc_HLM}. Now let $0\leq \beta < 1$ and $q = 1/(1-\beta)$. The following result is proved in \cite[Theorem 2]{CMa}: if $f \in BV(\Z)$ is such that $\wt{\M}_{\beta}f \not\equiv \infty$ then 
\begin{equation}\label{Intro_bound_disc_frac_HL}
\big\|\big(\wt{\M}_{\beta}f\big)'\big\|_{\ell^q(\Z)} = \var_q \big(\wt{\M}_{\beta}f\big)\leq 4^{1/q} \,\var(f).
\end{equation}
Regarding the continuity of this map we present the following negative result.
\begin{theorem}\label{Thm5}
Let $0 < \beta < 1$ and $q = 1/(1-\beta)$. The map $f \mapsto \big(\wt{\M}_{\beta}f\big)'$ is not continuous from $BV(\Z)$ to $\ell^q(\Z)$, i.e. there is a sequence $\{f_j\}_{j\geq 1} \subset BV(\Z)$ and a function $f \in BV(\Z)$ such that $\|f_j - f\|_{BV(\Z)} \to 0$ as $j \to \infty$ but 
$$\big\|\big(\wt{\M}_{\beta}f_j\big)' - \big(\wt{\M}_{\beta}f\big)'\big\|_{\ell^q(\Z)} = \var_q \big(\wt{\M}_{\beta}f_j - \wt{\M}_{\beta}f\big) \nrightarrow 0$$ 
as $j \to \infty$.
\end{theorem}

If we replace \eqref{Intro_Def_Disc_Max_oper_dim1} by its centered version $\M_{\beta}f$, the analogue of the bound \eqref{Intro_bound_disc_frac_HL} is still unknown. In spite of that, we can construct sequences that would violate a potential continuity.

\begin{theorem}\label{Thm6}
Let $0 < \beta < 1$ and $q = 1/(1-\beta)$. There is a sequence $\{f_j\}_{j\geq 1} \subset BV(\Z)$ and a function $f \in BV(\Z)$ such that $\|f_j - f\|_{BV(\Z)} \to 0$ as $j \to \infty$ but $\var_q \big(\M_{\beta}f_j - \M_{\beta}f\big) \nrightarrow 0$ as $j \to \infty$.
\end{theorem}

In parallel to Questions E and F about the boundedness and continuity in $W^{1,1}(\R)$ (whose discrete analogue is $\ell^1(\Z)$), if $0 < \beta < 1$ and $q = 1/(1-\beta)$ we remark that the maps $f \mapsto \big(\wt{\M}_{\beta}f\big)'$ and $f \mapsto \big(\M_{\beta}f\big)'$ are bounded and continuous from $\ell^1(\Z)$ to $\ell^q(\Z)$ as shown in \cite[Theorem 3]{CMa}.

\subsection{Summary} Throughout this introduction we have described a set of sixteen problems related to the endpoint continuity ($BV$ and $W^{1,1}$) for different versions of the one-dimensional Hardy-Littlewood maximal operator. Recall that these sixteen problems arise from the following pairs of possibilities: (i) centered vs. uncentered maximal operator; (ii) classical vs. fractional maximal operator; (iii) continuous vs. discrete setting; (iv) $W^{1,1}$ vs. $BV$ continuity.

\smallskip

Table 1 below summarizes the results of this paper and the open problems. The word YES in a box means that we have established the continuity of the corresponding map, whereas the word NO means that we have shown it fails. The remaining boxes are marked as OPEN problems.

\begin{table}[h]
\renewcommand{\arraystretch}{1.3}
\centering
\caption{Endpoint continuity program}
\label{Table-ECP}
\begin{tabular}{|c|c|c|c|c|}
\hline
 \raisebox{-1.3\height}{------------}&  \parbox[t]{2.8cm}{  $W^{1,1}-$continuity; \\ continuous setting} &  \parbox[t]{2.8cm}{  $BV-$continuity; \\ continuous setting} & \parbox[t]{2.6cm}{  $W^{1,1}-$continuity; \\ discrete setting} &  \parbox[t]{2.5cm}{  $BV-$continuity; \\ discrete setting} \\ [0.5cm]
\hline
 \parbox[t]{3.3cm}{ Centered classical \\ maximal operator} &  \raisebox{-0.8\height}{OPEN: Question A} & \raisebox{-0.8\height}{OPEN: Question C} & \raisebox{-0.8\height}{YES$^2$}  &\raisebox{-0.8\height}{OPEN: Question D}\\[0.5cm]
 \hline
  \parbox[t]{3.3cm}{ Uncentered classical \\ maximal operator} &  \raisebox{-0.8\height}{YES: Thm 1} &  \raisebox{-0.8\height}{OPEN: Question B} &  \raisebox{-0.8\height}{YES$^2$} &  \raisebox{-0.8\height}{YES: Thm 2} \\[0.5cm]
 \hline
 \parbox[t]{3.3cm}{ Centered fractional \\ maximal operator} &  \raisebox{-0.8\height}{OPEN$^1$: Question F} &  \raisebox{-0.8\height}{NO$^1$: Thm 4} &  \raisebox{-0.8\height}{YES$^3$} & \raisebox{-0.8\height}{NO$^1$: Thm 6}  \\[0.5cm]
 \hline
\parbox[t]{3.3cm}{ Uncentered fractional \\ maximal operator} &  \raisebox{-0.8\height}{OPEN: Question E} &  \raisebox{-0.8\height}{NO: Thm 3} &  \raisebox{-0.8\height}{YES$^3$}  & \raisebox{-0.8\height}{NO: Thm 5} \\[0.5cm]
 \hline
\end{tabular}
\vspace{0.05cm}
\flushleft{
\ \ \footnotesize{$^1$ Corresponding boundedness result not yet known.}\\
\ \ $^2$ Result previously obtained in \cite[Theorem 1]{CH}.\\
\ \ $^3$ Result previously obtained in \cite[Theorem 3]{CMa}.}
\end{table}

\section{$BV-$continuity in the discrete setting: proof of Theorem \ref{Thm2}}

We opt to consider the discrete cases first, since they describe the essence of the main ideas with fewer technicalities than the continuous cases. Throughout this section we denote the variation of $g:\Z \to \R$ over an interval $[a,b]$, with $a,b \in \Z \cup \{\pm \infty\}$, by
$$\var_{[a,b]}(g) :=\sum_{n = a}^{b-1} |g(n+1) - g(n)|.$$

\subsection{Preliminary lemmas}\label{sec_prelim_discrete}
We start by introducing the local maxima and minima of a discrete function $g:\Z \to \R$. We say that an interval $[m,n]$ is a {\it string of local maxima} of $g$ if 
$$g(m-1) < g(m) = \ldots = g(n) > g(n+1).$$
Here we allow the possibility $m=n$. We also allow the possibility of having $m = -\infty$ or $n = \infty$ (but not both simultaneously), and we modify the definition accordingly, eliminating one of the inequalities. The rightmost point $n$ (when $n \neq \infty$) of such a string is a {\it right local maximum} of $g$, while the leftmost point $m$ (when $m \neq -\infty$) is a {\it left local maximum} of $g$. We define {\it string of local minima}, {\it right local minimum} and {\it left local minimum} analogously. 

\smallskip

For $f \in BV(\Z)$ we let $\{[a_j^-,a_j^+]\}_{j \in \Z}$ and $\{[b_j^-,b_j^+]\}_{j \in \Z}$ be the ordered strings of local maxima and local minima of $\wt{\M}f$ (assuming $\wt{\M}f$ is non-constant), i.e. 
\begin{equation}\label{Sec2_Prel_lemmas_sequence_max_min}
\ldots < a_{-1}^- \leq a_{-1}^+ < b_{-1}^- \leq b_{-1}^+ < a_0^- \leq a_0^+ < b_0^-\leq b_0^+ < a_1^- \leq a_1^+ < b_1^- \leq b_1^+ < \ldots
\end{equation}

\begin{lemma}\label{lemma_min_max} If $\lim_{n\to -\infty}|f(n)| = a$ and $\lim_{n\to \infty}|f(n)| = b$, and $c = \max\{a,b\}$, then $\wt{\M}f(n) \geq c$ for all $n \in \Z$ and $\lim_{n \to \pm \infty} \wt{\M}f (n)= c$. 
In particular, if $\wt{\M}f$ is non-constant, this implies that for every string of local maxima $[a_j^-,a_j^+]$ both $a_j^-$ and $a_j^+$ are finite.
\end{lemma}
 We may still have $b_j^- = -\infty$ or $b_j^+ = \infty$ for some $j \in \Z$ in \eqref{Sec2_Prel_lemmas_sequence_max_min}.

\begin{proof} 
For the first statement, if without loss of generality we assume that $c=b$, then for any fixed $n$ we have $\wt{\M}f(n) \geq (r+1)^{-1} \sum_{0 \leq k \leq r} |f(n+k)|$, and by choosing $r$ arbitrarily large we may make the averages arbitrarily close to $c$, so that $\wt{\M}f(n) \geq c$. To show that in the limit $\wt{\M}f(n) = c$, we suppose on the contrary that there exists an infinite sequence $n_j$ such that $\wt{\M}f(n_j) \geq c+\delta$ for some fixed $\delta>0$. Now fix $x,y$ such that 
\begin{equation}\label{comparison}
\text{ for all $n \in (-\infty ,x]$, $\big||f(n)| - a\big| \leq \delta/8$ and for all $n \in [y,\infty)$, $\big||f(n)| - b\big| \leq \delta/8$.}
\end{equation}
 We may then choose $j_0$ sufficiently large that for each $j \geq j_0$, $n_j \not\in[x,y]$ and moreover is sufficiently distant from $[x,y]$ (relative to $\|f\|_{\ell^\infty}$) that for every interval $[n_j-r,n_j+s]$,
\begin{equation}\label{cap_small}
\frac{1}{(r+s+1)} \sum_{u \in [x,y] \cap [n_j-r,n_j+s]} |f(u)|  \leq \delta/4.
\end{equation}
Then for $j \geq j_0$, the statement $\wt{\M}f(n_j) \geq c+\delta$ would imply that there exist $r,s$ such that 
\[ c+\delta/2 \leq \frac{1}{(r+s+1)} \sum_{u \in [n_j-r,n_j+s]} \!\!\! \!\!  \!\!|f(u)|  = \frac{1}{(r+s+1)} \sum_{u \in [x,y] \cap [n_j-r,n_j+s]} \!\!\!\!\! \!\!|f(u)|
	\,+\,\frac{1}{(r+s+1)} \sum_{u \in [n_j-r,n_j+s] \setminus [x,y]} \!\!\!\!\! \!\! |f(u)| .\]
By (\ref{cap_small}) this shows that
\[\frac{1}{(r+s+1)} \sum_{u \in [n_j-r,n_j+s] \setminus [x,y]} |f(u)| \geq c+\delta/4\,,\]
which is not possible, since for all $u$ in this sum, $|f(u)| \leq c+\delta/8$, by (\ref{comparison}). Thus we may conclude that $\lim_{n \rightarrow \pm\infty} \wt{\M}f(n)=c$.	

\smallskip

To prove the final statement in the lemma, if on the contrary we had (for example) $a_j^+=\infty$, then upon setting $c =\wt{\M}f(a_j^-)$ we would have $\wt{\M}f(n)=c$ for all $n \geq a_j^-$, so that in the notation above $c=\max\{a,b\}$, and by the first part of the lemma, $\wt{\M}f(n) \geq c$ for all $n$ and $\lim_{n \rightarrow \pm \infty} \wt{\M}f(n) =c$. Then either $a_j^-=-\infty$ in which case $\wt{\M}$ is constant, or $a_j^{-}$ is finite and hence $\wt{\M}f(a_j^- -1) < \wt{\M}f(a_j^-)$. But this cannot happen, since $\wt{\M}f(n) \geq c$ for all $n$. 
  The argument to rule out $a_j^-=-\infty$ is similar.
  \end{proof}

Using the notation of (\ref{Sec2_Prel_lemmas_sequence_max_min}) we may express the variation as
$$\var\big(\wt{\M}f\big) = \sum_{j = -\infty}^{\infty} \var_{[a_j^+, a_{j+1}^-]}\big(\wt{\M}f\big) = 
\sum_{j = -\infty}^{\infty} \, \Big\{ \wt{\M}f(a_j^+) - \wt{\M}f(b_j^-)\Big\} + \Big\{ \wt{\M}f(a_{j+1}^-) - \wt{\M}f(b_j^+)\Big\}.$$
Trivial modifications apply to the formula above in case there is a last string of local maxima to one of the sides (in this case, the side limits $\lim_{n\to \pm \infty} \wt{\M}f(n)$ would appear in the formula), and we adjust the notation accordingly. 

\smallskip

It was proved in \cite[Lemma 3]{BCHP} that if $f \in BV(\Z)$ and $n$ is a point of right or left local maximum of $\wt{\M}f$, then 
\begin{equation}\label{point_of_contact}
\wt{\M}f(n) = |f(n)|.
\end{equation}
This observation leads to the following lemma.

\begin{lemma}[One-sided control]\label{Lem7}
Let $f \in BV(\Z)$ and let  $\{[a_j^-,a_j^+]\}_{j \in \Z}$ and $\{[b_j^-,b_j^+]\}_{j \in \Z}$ be the ordered strings of local maxima and local minima of $\wt{\M}f$ as in \eqref{Sec2_Prel_lemmas_sequence_max_min}. The following statements hold:
\begin{enumerate}
\item[(i)] If $b_i^- \leq n \leq a_{i+1}^+$ for some $i \in \Z$ $($or if $n \leq a_{i+1}^+$ and $\wt{\M}f$ is non-decreasing on $(-\infty, a_{i+1}^+]$$)$, then
$$\var_{[n,\infty)}\big(\wt{\M}f\big) \leq \var_{[n,\infty)}(f).$$

\item[(ii)] If $a_i^- \leq n \leq b_{i}^+$ for some $i \in \Z$ $($or if $a_i^- \leq n$ and $\wt{\M}f$ is non-increasing on $[a_{i}^-, \infty)$$)$, then
$$\var_{(-\infty,n]}\big(\wt{\M}f\big) \leq \var_{(-\infty,n]}(f).$$

\end{enumerate}
\end{lemma}

The previous lemma describes the {\it one-sided control} of the variation of the uncentered maximal function by the variation of the original function, and is hence a refinement of the previously obtained bound (\ref{Intro_BCHP}). Such control is unknown for the centered maximal function (at least it is not apparent from the proofs of the $BV-$boundedness in \cite{Ku, Te}), and it is the main reason why our argument for the $BV-$continuity does not apply to the centered case. We now indicate the proof of the Lemma \ref{Lem7}.
\begin{proof}
We consider the case (i). First supposing $b_i^- \leq n \leq a_{i+1}^+$ then 
\begin{align}\label{Eq_1_pf_Lem8}
\begin{split}
\var_{[n,\infty)}\big(\wt{\M}f\big) 
	&= (\wt{\M}f(a_{i+1}^+) - \wt{\M}f(n)) +  \sum_{j = i+1}^{\infty} \var_{[a_j^+, a_{j+1}^-]}\big(\wt{\M}f\big) \\
	&= 
(\wt{\M}f(a_{i+1}^+) - \wt{\M}f(n)) + \sum_{j = i+1}^{\infty} \, \Big\{ \wt{\M}f(a_j^+) - \wt{\M}f(b_j^-)\Big\} + \Big\{ \wt{\M}f(a_{j+1}^-) - \wt{\M}f(b_j^+)\Big\}.
\end{split}
\end{align}
At each $a_j^+$ or $a_j^-$ we apply (\ref{point_of_contact}), while at each $b_j^-$ or $b_j^+$ we apply the trivial bound $\wt{\M}f(n) \geq |f(n)|$; as a result we conclude 
\[ 
\var_{[n,\infty)}\big(\wt{\M}f\big)  \leq 
\big(|f(a_{i+1}^+)| - |f(n)|\big) + \sum_{j = i+1}^{\infty} \, \big\{ |f(a_j^+)| - |f(b_j^-)|\big\} + \big\{ |f(a_{j+1}^-)| - |f(b_j^+)|\big\} \leq \var_{[n,\infty)} (f),
\]
as claimed. In case there is a last string of local maxima, minor modifications apply to \eqref{Eq_1_pf_Lem8} and we use Lemma \ref{lemma_min_max}. In the case that $n \leq a_{i+1}^+$ and $\wt{\M}f$ is non-decreasing on $(-\infty, a_{i+1}^+]$$)$, the above argument also applies.
Part (ii) follows analogously.
\end{proof}

\smallskip

Our next lemma establishes the pointwise convergence of the maximal functions.

\begin{lemma}[Pointwise convergence]\label{Lem8}
Let $\{f_j\}_{j\geq 1} \subset BV(\Z)$ and $f \in BV(\Z)$ be such that $\|f_j - f\|_{BV(\Z)} \to 0$ as $j \to \infty$. Then $\wt{\M} f_j(n) \to \wt{\M} f (n)$ uniformly for $n \in \Z$.
\end{lemma}

\begin{proof}
Given $\varepsilon >0$, there exists $j_0$ such that 
$$\|f_j - f\|_{\ell^\infty(\Z)} \leq \|f_j - f\|_{BV(\Z)} < \varepsilon$$ 
for $j \geq j_0$. By the sublinearity of the maximal operator we find that 
$$\big|\wt{\M} f_j(n) - \wt{\M} f(n)\big| \leq \big|\wt{\M} (f_j - f) (n)\big| < \varepsilon$$
for $j \geq j_0$, uniformly in $n$.
\end{proof}

The next result is the Brezis-Lieb reduction we need in order to prove Theorem \ref{Thm2}. 

\begin{lemma}[Brezis-Lieb reduction]
Let $\{f_j\}_{j\geq 1} \subset BV(\Z)$ and $f \in BV(\Z)$ be such that $\|f_j - f\|_{BV(\Z)} \to 0$ as $j \to \infty$. In order to prove Theorem \ref{Thm2} it suffices to show that 
\begin{equation}\label{Lem4_eq1_cond}
\var\big(\wt{\M}f_j\big) \to \var\big(\wt{\M}f\big)
\end{equation}
as $j \to \infty$.
\end{lemma}
\begin{proof}
For any fixed $n$ we have
$$\big(\wt{\M}f_j\big)'(n) = \big(\wt{\M}f_j - \wt{\M}f\big)'(n)  + \big(\wt{\M}f\big)'(n)$$
and thus, by triangle inequality,
$$\Big| \big|  \big(\wt{\M}f_j\big)'(n) \big| - \big| \big(\wt{\M}f_j - \wt{\M}f\big)'(n) \big| \Big| \leq \big|  \big(\wt{\M}f\big)'(n) \big|.$$
Adding up and applying (\ref{Intro_BCHP}) we get 
\begin{equation}\label{cond_Dom_conv}
\sum_{n \in \Z} \Big| \big|  \big(\wt{\M}f_j\big)'(n) \big| - \big| \big(\wt{\M}f_j - \wt{\M}f\big)'(n) \big| \Big| \leq \sum_{n \in \Z}\big|  \big(\wt{\M}f\big)'(n) \big| = \var\big(\wt{\M}f\big) \leq \var (f) < \infty.
\end{equation}
By \eqref{cond_Dom_conv}, we may apply the dominated convergence theorem and Lemma \ref{Lem8} to get
\begin{align}\label{Lem3_eq1}
\begin{split}
 \lim_{j \to \infty} \sum_{n \in \Z} & \Big( \big|  \big(\wt{\M}f_j\big)'(n) \big| -  \big| \big(\wt{\M}f_j -  \wt{\M}f\big)'(n) \big| \Big) \\
 & = \sum_{n \in \Z}  \lim_{j \to \infty}  \Big( \big|  \big(\wt{\M}f_j\big)'(n) \big| -  \big| \big(\wt{\M}f_j -  \wt{\M}f\big)'(n) \big|  \Big) =  \var\big(\wt{\M}f\big).
 \end{split}
\end{align}

Now assume that \eqref{Lem4_eq1_cond} holds. Then
\begin{equation}\label{Lem3_eq2}
 \lim_{j \to \infty} \sum_{n \in \Z} \Big(  \big|  \big(\wt{\M}f_j\big)'(n) \big| -  \big| \big(\wt{\M}f_j -  \wt{\M}f\big)'(n) \big| \Big) =  \var\big(\wt{\M}f\big) -  \lim_{j \to \infty} \var\big(\wt{\M}f_j - \wt{\M}f\big).
\end{equation}
From \eqref{Lem3_eq1} and \eqref{Lem3_eq2} we conclude that 
\begin{equation}\label{Sec2_Lem9_conc1}
\lim_{j \to \infty} \var\big(\wt{\M}f_j - \wt{\M}f\big) = 0.
\end{equation}
Let $\lim_{n\to -\infty}|f_j(n)| = a_j$, $\lim_{n\to \infty}|f_j(n)| = b_j$ and $c_j = \max\{a_j,b_j\}$, as well as $\lim_{n\to -\infty}|f(n)| = a$, $\lim_{n\to \infty}|f(n)| = b$ and $c= \max\{a,b\}$. Then assuming $\|f_j - f\|_{BV(\Z)} \to 0$ as $j \to \infty$ immediately implies that $a_j \to a$, by the definition of the BV norm. We may also then easily deduce that $b_j \to b$ and hence $c_j \to c$ as $j \to \infty$. Indeed, if $b_j \to b$ were false, then we would find an infinite subsequence of $j_i$ such that $|b_{j_i} - b| \geq \delta$ for some fixed $\delta>0$; we may further assume that we consider only $j_i$ sufficiently large that $|a_{j_i} - a| \leq \delta/4$. Take $n_0$ sufficiently large that for all $n \geq n_0$, $|f(n) - b| \leq \delta/4$. For each $j_i$ in the subsequence take $n_{j_i} \geq n_{j_{i-1}}$ such that for all $n \geq n_{j_i}$, $|f_{j_i}(n) - b_{j_i}| \leq \delta/4$.  
Then for each $i$ we would have
\[ 
\var (f - f_{j_i}) \geq \big|(f-f_{j_i})(n_{j_i})\big| - \big|(f - f_{j_i})(-\infty)\big| = \big|(f-f_{j_i})(n_{j_i})\big|  - |a-a_{j_i}| \geq \delta/4,
\]
which contradicts $\|f- f_j\|_{BV(\Z)} \to 0.$ Thus we may conclude that $b \to b_j$, as desired.

\smallskip

Finally, by Lemma \ref{lemma_min_max}
\begin{equation}\label{Sec2_Lem9_conc2}
\lim_{n \to -\infty} \wt{\M}f_j(n) = c_j \to c  = \lim_{n \to -\infty} \wt{\M}f(n)\,,
\end{equation}
and we conclude from \eqref{Sec2_Lem9_conc1} and \eqref{Sec2_Lem9_conc2} that $\big\|\wt{\M}f_j - \wt{\M}f\big\|_{BV(\Z)} \to 0$ as $j \to \infty$.
\end{proof}

\subsection{Proof of Theorem \ref{Thm2}} We shall prove \eqref{Lem4_eq1_cond}. By Lemma \ref{Lem8} and Fatou's lemma we have
\begin{align*}
\var\big(\wt{\M}f\big) = \sum_{n \in \Z} \liminf_{j \to \infty} \big|\big(\wt{\M}f_j\big)'(n)\big| \leq \liminf_{j \to \infty} \sum_{n \in \Z} \big|\big(\wt{\M}f_j\big)'(n)\big|  = \liminf_{j \to \infty} \,\var\big(\wt{\M}f_j\big) .
\end{align*}
We are then left to prove that 
\begin{equation}\label{Thm1_eq_limsup}
\limsup_{j \to \infty} \,\var\big(\wt{\M}f_j\big) \leq \var\big(\wt{\M}f\big).
\end{equation}
Given $\delta > 0$, choose finite integers $x<y$ such that 
$$\var_{(-\infty, x]}\big(\wt{\M}f\big)< \delta \ \ ; \ \  \var_{[y, \infty)}\big(\wt{\M}f\big) < \delta  \ \ ; \ \ \var_{(-\infty, x]}(f) < \delta \ \ ;  \ \ \var_{[y, \infty)}(f) < \delta.$$
Note that here we have used that $\var \big(\wt{\M}f\big)$ is finite, since $\var(f)$ is, and  (\ref{Intro_BCHP}) is known. Since $\wt{\M}f_j \to \wt{\M}f$ uniformly, we can choose $j_0$ such that for $j \geq j_0$ we have 
$$\big\|\wt{\M}f_j - \wt{\M}f\big\|_{\ell^{\infty}(\Z)} < \delta  \ \ ; \ \ \var_{[x,y]}\big(\wt{\M}f_j - \wt{\M}f\big) < \delta  \ \ \ ; \ \  \|f_j - f\|_{BV(\Z)} < \delta.$$
Now observe that, for $j \geq j_0$,  
\begin{align}
\begin{split}\label{Thm1_eq0_main_stream}
\var\big(\wt{\M}f_j\big) & = \var_{(-\infty, x]}\big(\wt{\M}f_j\big)+  \var_{[x, y]}\big(\wt{\M}f_j\big)  +  \var_{[y, \infty)}\big(\wt{\M}f_j\big) \\
& \leq  \var_{(-\infty, x]}\big(\wt{\M}f_j\big)  + \Big(\var_{[x, y]}\big(\wt{\M}f\big) + \var_{[x, y]}\big(\wt{\M}f_j - \wt{\M}f\big) \Big) +  \var_{[y, \infty)}\big(\wt{\M}f_j\big) \\
& \leq  \var_{(-\infty, x]}\big(\wt{\M}f_j\big)  + \Big(  \var \big(\wt{\M}f\big)  + \delta \Big) +   \var_{[y, \infty)}\big(\wt{\M}f_j\big).
\end{split}
\end{align}

\smallskip

Let us prove that $ \var_{[y, \infty)} \big(\wt{\M}f_j\big) $ is small. Fix $j \geq j_0$ and label the strings of local extrema of $\wt{\M}f_j$ as in \eqref{Sec2_Prel_lemmas_sequence_max_min} (here we suppress the dependence on $j$ for simplicity, and assume that $\wt{\M}f_j$ is not constant, since otherwise the verification is trivial):
\begin{equation*}
\ldots < a_{-1}^- \leq a_{-1}^+ < b_{-1}^- \leq b_{-1}^+ < a_0^- \leq a_0^+ < b_0^-\leq b_0^+ < a_1^- \leq a_1^+ < b_1^- \leq b_1^+ < \ldots
\end{equation*}
Let us divide the analysis into three cases:

\subsubsection*{Case 1} If $b_i^- \leq y \leq a_{i+1}^+$ for some $i \in \Z$ (or if $y \leq a_{i+1}^+$ and $\wt{\M}f_j$ is non-decreasing on $(-\infty, a_{i+1}^+]$). By Lemma \ref{Lem7}(i) we have
\begin{align*}
 \var_{[y, \infty)}\big(\wt{\M}f_j\big) \leq  \var_{[y, \infty)}(f_j) \leq  \var_{[y, \infty)}(f) +  \var_{[y, \infty)}(f_j - f) < 2\delta.
\end{align*} 

\subsubsection*{Case 2} If $a_i^+ < y < b_{i}^-$ for some $i \in \Z$. Then (applying Lemma \ref{Lem7}(i) in the first inequality),
\begin{align*}
 \var_{[y, \infty)}\big(\wt{\M}f_j\big) &  =  \var_{[y, b_{i}^-]}\big(\wt{\M}f_j\big) + \var_{[ b_{i}^-, \infty)}\big(\wt{\M}f_j\big)\\
 & \leq \Big(\wt{\M}f_j(y) - \wt{\M}f_j(b_{i}^-)\Big) +  \var_{[ b_{i}^-, \infty)}(f_j)\\
 & \leq \Big(\big(\wt{\M}f(y) + \delta\big) - \big(\wt{\M}f(b_{i}^-) - \delta\big)\Big) + \var_{[ b_{i}^-, \infty)}(f) + \var_{[ b_{i}^-,  \infty)}(f_j-f) \\
 & < \Big( \var_{[y, b_{i}^-]}\big(\wt{\M}f\big) + 2\delta\Big) + 2\delta\\
 & < 5\delta.
\end{align*}

\subsubsection*{Case 3} If $a_i^+ < y$ and  $a_i^+ $ is the last term in the sequence \eqref{Sec2_Prel_lemmas_sequence_max_min}. This means that $\wt{\M}f_j$ is non-increasing on $[y, \infty)$. If we let $c_j = \lim_{n \to \infty} \wt{\M}f_j(n)$ and $c = \lim_{n \to \infty} \wt{\M}f(n)$, then $|c_j - c| < \delta$. Thus
\begin{align*}
 \var_{[y, \infty)}\big(\wt{\M}f_j\big) = \wt{\M}f_j(y) - c_j \leq \big(\wt{\M}f(y) + \delta\big) - (c - \delta) \leq \var_{[y, \infty)}\big(\wt{\M}f\big) + 2\delta < 3\delta.
\end{align*}

We have thus covered all the cases with the analysis above, arriving at
\begin{equation}\label{Thm1_eq1_right}
\var_{[y, \infty)}\big(\wt{\M}f_j\big) < 5\delta
\end{equation}
for $j \geq j_0$. Proceeding analogously using Lemma \ref{Lem7}(ii), we get
\begin{equation}\label{Thm1_eq1_left}
\var_{(-\infty, x]}\big(\wt{\M}f_j\big) < 5\delta
\end{equation}
for $j \geq j_0$. Plugging \eqref{Thm1_eq1_right} and \eqref{Thm1_eq1_left} into \eqref{Thm1_eq0_main_stream} we arrive at
$$\var\big(\wt{\M}f_j\big)  < \var\big(\wt{\M}f\big) + 11\delta$$
for $j \geq j_0$. Since our parameter $\delta >0$ was arbitrary, this establishes \eqref{Thm1_eq_limsup} and concludes the proof.

\section{Counterxamples in the discrete setting: proofs of Theorems \ref{Thm5} and \ref{Thm6}}

\subsection{Proof of Theorem \ref{Thm5}} For $g:\Z \to \R$, $n\in\Z$ and $r,s \in \Z^+ \times \Z^+$, we define the fractional average
$$
\mc{A}_{r,s}g(n):=\frac{1}{(r+s+1)^{1-\beta}}\sum_{k=-r}^{s}g(n+k).
$$
Let $f: \Z \to \R$ be the discrete delta function at the origin, i.e.
\begin{equation}\label{pf_Thm5_delta_function}
f(0)=1\ \ \text{and} \ \ f(n)=0 \ \ \forall \ n\in\Z\setminus\{0\}. 
\end{equation}
For $0 < \beta < 1$ and $\wt{\M}_{\beta}f$ defined as in \eqref{Intro_Def_Disc_Max_oper_dim1}, note that 
\begin{equation}\label{eq_pf_thm5_f}
\wt{\M}_{\beta}f(0)=1,\ \ \wt{\M}_{\beta}f(1)=2^{\beta-1}\  \text{and}\  \ \big(\wt{\M}_{\beta}f\big)'(0)=2^{\beta-1}-1\neq 0.
\end{equation}

\smallskip

For each $j\in\N$ define the function $L_j:\Z^+ \to \R$ by
\begin{equation*}
L_j(s) = \frac{1}{(s+1)^{1-\beta}}\left(\frac{s+1}{2j}+1\right).
\end{equation*}
Note that for each $j$, $L_j(0) = 1 + \tfrac{1}{2j}$ and $L_j(s) \to \infty$ as $s \to \infty$ due to the fact that $0 <\beta$. We now construct inductively an increasing sequence $\{h_j\}_{j \geq1}$ of natural numbers by letting $h_j$ be the smallest integer such that $h_j > h_{j-1}$ ($h_1 \geq 1$) and
\begin{equation}\label{condition_L_pf_Thm5}
L_j(h_j) > L_j(s)
\end{equation}
for all $s=0, 1,2,\ldots, h_j -1$.

\smallskip

For each $j \in \N$ define $f_j: \Z \to \R$ by 
\begin{equation}\label{def_f_j_counter_Thm5}
f_{j}(n):=f(n)+\frac{1}{2j}\chi_{[0,h_{j}]}(n)\ \ \forall \ n\in\Z,
\end{equation}
where $\chi_{[0,h_{j}]}$ denotes the characteristic function of the interval $[0,h_{j}]$. Observe that 
$$\|f_j - f\|_{BV(\Z)} = \frac{1}{j} \to 0$$ 
as $j \to \infty$. From \eqref{condition_L_pf_Thm5} we have that $\wt{\M}_{\beta}f_j(0)= \mc{A}_{0,h_j}f_j(0)$. At the point $n=1$, there is a pair of nonnegative integers $(r_j,s_j)$ such that $\wt{\M}_{\beta}f_j(1) = \mc{A}_{r_j,s_j}f_j(1)$. From \eqref{condition_L_pf_Thm5} we have that $(r_j,s_j) \neq (0,0)$ (since $\mc{A}_{1,h_j-1}f_j(1) > \mc{A}_{0,0}f_j(1)$). Also, we cannot have $r_j = 0$, for in this case we would have $s_j \geq 1$ and $\mc{A}_{1,s_j-1}f_j(1) > \mc{A}_{0,s_j}f_j(1)$, a contradiction. Hence $r_j \geq 1$ and another application of \eqref{condition_L_pf_Thm5} gives us that $\wt{\M}_{\beta}f_j(1)= \mc{A}_{1,h_j-1}(1) = \wt{\M}_{\beta}f_j(0)$. We conclude that for all $j \in \N$,
\begin{equation}\label{eq_pf_thm5_f_j}
\big(\wt{\M}_{\beta}f_j\big)'(0) = 0.
\end{equation}
From \eqref{eq_pf_thm5_f} and \eqref{eq_pf_thm5_f_j} it follows that 
$$\big\|\big(\wt{\M}_{\beta}f_j\big)' - \big(\wt{\M}_{\beta}f\big)'\big\|_{\ell^q(\Z)} = \var_q \big(\wt{\M}_{\beta}f_j - \wt{\M}_{\beta}f\big) \nrightarrow 0$$ 
as $j \to \infty$.

\subsection{Proof of Theorem \ref{Thm6}} For $g:\Z \to \R$, $n\in\Z$ and $r \in \Z^+$, we define the (centered) fractional average
$$
\mc{A}_{r}g(n):=\frac{1}{(2r+1)^{1-\beta}}\sum_{k=-r}^{r}g(n+k).
$$
Let $f:\Z \to \R$ be the discrete delta function as in \eqref{pf_Thm5_delta_function}. Then 
\begin{equation}\label{eq_pf_thm6_f_centered}
\M_{\beta}f(0)=1,\ \ \M_{\beta}f(1)=3^{\beta-1}\  \text{and}\  \ \big(\M_{\beta}f\big)'(0)=3^{\beta-1}-1\neq 0.
\end{equation}

We proceed similarly as before, conveniently constructing an increasing sequence $\{h_j\}_{j \geq1}$ of natural numbers such that $f_j$ defined as in \eqref{def_f_j_counter_Thm5} verifies
\begin{equation}\label{cond_pf_Thm6}
\M_{\beta}f_j(0)= \mc{A}_{h_j} f_j(0) \ \ {\rm and} \ \  \M_{\beta}f_j(1)= \mc{A}_{h_j-1}f_j(1).
\end{equation}
If we have \eqref{cond_pf_Thm6}, then 
\begin{align}\label{pf_thm6_eq2}
\begin{split}
\big(\M_{\beta}f_j)'(0) & = \M_{\beta}f_j(1) - \M_{\beta}f_j(0)\\
& = \left(1 + \frac{h_j+1}{2j}\right)\left( (2h_j-1)^{\beta -1} - (2h_j+1)^{\beta -1}\right)\\
& \sim \left(1 + \frac{h_j+1}{2j}\right) h_j^{\beta -2}\\
& \to 0
\end{split}
\end{align}
as $j \to \infty$, since $0 < \beta < 1$. Then, from \eqref{eq_pf_thm6_f_centered} and \eqref{pf_thm6_eq2} we conclude that 
$$\big\|\big(\M_{\beta}f_j\big)' - \big(\M_{\beta}f\big)'\big\|_{\ell^q(\Z)} = \var_q \big(\M_{\beta}f_j - \M_{\beta}f\big) \nrightarrow 0$$ 
as $j \to \infty$.

\smallskip

To construct $\{h_j\}_{j \geq1}$ verifying \eqref{cond_pf_Thm6} just observe that, for each $j$, the functions 
$$F_j(r ) = \frac{1}{(2r+1)^{1-\beta}}\left(\frac{r+1}{2j}+1\right) \ \ {\rm and} \ \ G_j( r)= \frac{1}{(2r+1)^{1-\beta}}\left(\frac{r+2}{2j}+1\right)$$
both tend to infinity as $r\to \infty$, and are strictly increasing for $r \geq r_0(j,\beta)$. We can define $h_j>h_{j-1}$ to be the smallest integer such that $F_j(h_j) > F_j(s)$ for all $s=0,1,2, \ldots, h_{j}-1$ and $G_j(h_j-1) > G_j(s)$, for all $s=0,1,2, \ldots, h_{j}-2$. Assertion (\ref{cond_pf_Thm6}) can then be verified.

\section{Counterxamples in the continuous setting: proofs of Theorems \ref{Thm3} and \ref{Thm4}}

We proceed here by adapting the counterexamples of the discrete setting, presented in the previous section, to the continuous setting.

\subsection{Proof of Theorem \ref{Thm3}} For $g \in L^1_{loc}(\R)$, $x \in \R$ and $r,s \geq 0$, we define the fractional average
$$
A_{r,s}g(x):=\frac{1}{(r+s)^{1-\beta}}\int_{x-r}^{x+s} g(y)\,\dy.
$$
Since $0 < \beta < 1$, note that $A_{0,0}g(x) = 0$.

\smallskip

Let $f(x) = \chi_{[0,1]}(x)$. One can check that 
\begin{equation}\label{pf_Thm3_eq1}
\wt{M}_{\beta}f(0) = A_{0,1}f(0) = 1 \ \ {\rm and} \ \  \wt{M}_{\beta}f(2) = A_{2,0}f(2) = 2^{\beta-1}.
\end{equation}
For $j \in \N$ we define 
\begin{equation}\label{pf_Thm3_def_f_j}
f_j (x) := f(x)+ \frac{1}{2j}\chi_{[0,h_{j}]}(x),
\end{equation}
where $\{h_j\}_{j \geq1}$ is an increasing sequence of natural numbers (say, greater than $2$) to be conveniently chosen later. Observe that $\|f_j - f\|_{BV(\R)} = \frac{1}{j} \to 0$. We want to choose $h_j$ in a way that 
\begin{equation}\label{pf_Thm3_eq2}
\wt{M}_{\beta}f_j(0) = \wt{M}_{\beta}f_j(2) = A_{0,h_j}f_j(0).
\end{equation}
In light of \eqref{pf_Thm3_eq1} and \eqref{pf_Thm3_eq2} we would have
\begin{align*}
\var_q \big(\wt{M}_{\beta}f_j - \wt{M}_{\beta}f\big) \geq \left( \frac{\big|\big(\wt{M}_{\beta}f_j - \wt{M}_{\beta}f\big)(2) -\big(\wt{M}_{\beta}f_j - \wt{M}_{\beta}f\big)(0)\big|^q}{|2-0|^{q-1}}\right) \nrightarrow 0
\end{align*}
as $j \to \infty$.

\smallskip

In order to properly choose $\{h_j\}_{j \geq1}$ observe that, for each $j \in \N$, the function 
$$F_j(s) = \frac{1}{s^{1-\beta}} \left( 1 + \frac{s}{2j}\right)$$
is such that $F_j(s) \to \infty$ as $s \to \infty$, and $F_j$ is strictly increasing for $s \geq s_0(j,\beta)$. We can inductively construct an increasing sequence $\{h_j\}_{j \geq1}$ of natural numbers (greater than $2$) such that 
$$\wt{M}_{\beta}f_j(0)  = A_{0,h_j}f_j(0) > A_{0,s}f_j(0)$$ 
for all $j \in \N$ and all $0 \leq s < h_j$. With such a choice, observe that we must also have $\wt{M}_{\beta}f_j(2) = A_{0,h_j}f_j(0)$. In fact, for each $j \in \N$, there exists a pair $(r_j, s_j) \neq (0,0)$ such that $\wt{M}_{\beta}f_j(2) = A_{r_j,s_j}f_j(2)$. Clearly $r_j \leq 2$ and $s_j \leq h_j -2$. Also, we must have $s_j >0$ since $A_{2,h_j-2}f_j(2) > A_{2,0}f_j(2) > A_{r,0}f_j(2)$ for any $0\leq r <2$. The fact that $A_{r+\varepsilon, s - \varepsilon}f_j(2) > A_{r,s}f_j(2)$ if $r<2$, $0 < s$ and $0 < \varepsilon < \min\{2-r, s\}$ implies that we must have $r_j =2$. Finally, since by construction $A_{2,h_j-2}f_j(2) > A_{2,s}f_j(2)$ for any $0 \leq s < h_j -2$, we conclude that $s_j = h_j-2$ and $\wt{M}_{\beta}f_j(2) = A_{2,h_j-2}f_j(2) = A_{0,h_j}f_j(0)= \wt{M}_{\beta}f_j(0)$. This concludes the proof.

\subsection{Proof of Theorem \ref{Thm4}} For $g \in L^1_{loc}(\R)$, $x \in \R$ and $r \geq 0$, define now the centered fractional average
$$
A_{r}g(x):=\frac{1}{(2r)^{1-\beta}}\int_{x-r}^{x+r} g(y)\,\dy.
$$
Since $0 < \beta < 1$, note again that $A_{0}g(x) = 0$. As before, let $f(x) = \chi_{[0,1]}(x)$ and, for $j \in \N$, define $f_j$ as in \eqref{pf_Thm3_def_f_j}, with $\{h_j\}_{j \geq1}$ being an increasing sequence of natural numbers (say, greater than $2$) to be chosen later. One can check that 
\begin{equation}\label{pf_Thm4_eq1}
M_{\beta}f(0) = A_{1}f(0) = 2^{\beta-1} \ \ {\rm and} \ \  M_{\beta}f(2) = A_{2}f(0) = 4^{\beta-1}.
\end{equation}
We will choose the sequence $\{h_j\}_{j \geq1}$ in a way that 
\begin{equation}\label{pf_Thm4_eq2}
M_{\beta}f_j(0) = A_{h_j}f_j(0) \ \ {\rm and} \ \ M_{\beta}f_j(2) = A_{h_j-2}f_j(2),
\end{equation}
and in this case, observe that
\begin{align}\label{pf_Thm4_eq3}
\begin{split}
M_{\beta}f_j(2) - M_{\beta}f_j(0) & = \left(1 + \frac{h_j}{2j}\right)\left( (2h_j-4)^{\beta -1} - (2h_j)^{\beta -1}\right)\\
& \sim \left(1 + \frac{h_j}{2j}\right) h_j^{\beta -2}\\
 & \to 0
\end{split}
\end{align}
as $j \to \infty$. In light of \eqref{pf_Thm4_eq1} and \eqref{pf_Thm4_eq3} we would have
\begin{align*}
\var_q \big(M_{\beta}f_j - M_{\beta}f\big) \geq \left( \frac{\big|\big(M_{\beta}f_j - M_{\beta}f\big)(2) -\big(M_{\beta}f_j - M_{\beta}f\big)(0)\big|^q}{|2-0|^{q-1}}\right) \nrightarrow 0
\end{align*}
as $j \to \infty$.

\smallskip

It remains to choose the increasing sequence $\{h_j\}_{j \geq1}$ of natural numbers in order to have \eqref{pf_Thm4_eq2}. This can be accomplished by observing that, for each $j \in \N$, the functions 
$$F_j(r ) = \frac{1}{(2r)^{1-\beta}}\left(\frac{r}{2j}+1\right) \ \ {\rm and} \ \ G_j( r)= \frac{1}{(2r)^{1-\beta}}\left(\frac{r+2}{2j}+1\right)$$
both tend to infinity as $r\to \infty$, and are strictly increasing for $r \geq r_0(j,\beta)$.

\section{$W^{1,1}-$continuity: proof of Theorem \ref{Thm1}}

We now come to the final, and perhaps the most delicate, part of the paper: the proof of Theorem \ref{Thm1}. We start with some auxiliary results, some of which may be of independent interest. Recall that if $f \in W^{1,1}(\R)$ then, after adjusting on a set of measure zero, $f$ {\it may be taken to be absolutely continuous $($and we always assume this is the case in this section$)$}.

\subsection{Preliminaries I - Reduction to one-sided maximal functions} For $f \in L^1_{loc}(\R)$ we define the one-sided maximal functions
$$
M_{R}f(x)=\sup_{r>0}\frac{1}{r}\int_{x}^{x+r}|f(y)|\,\dy\ \ \ \text{and}\ \ \ M_{L}f(x)=\sup_{r>0}\frac{1}{r}\int_{x-r}^{x}|f(y)|\,\dy.
$$
One observes (see \cite[Eqn. (5)]{Ta}) that
\begin{equation}\label{Expression_M_MR_ML}
\widetilde Mf(x)=\max\big\{ M_{R}f(x), M_{L}f(x)\big\}\ \ \ \forall x\in\R.
\end{equation}
Recall that if $f \in W^{1,1}(\R)$ (or even just $f \in BV(\R)$) then the uncentered maximal function $\widetilde Mf$ is absolutely continuous and
\begin{equation*}
\big\|\big(\wt{M}f\big)'\big\|_{L^1(\R)} \leq  \|f'\|_{L^1(\R)}.
\end{equation*}
This was proved by Aldaz and P\'{e}rez L\'{a}zaro in \cite[Theorem 2.5]{AP}. In \cite{Ta}, Tanaka observed that if $f \in W^{1,1}(\R)$, then $M_Rf$ and $M_Lf$ are continuous functions that are weakly differentiable and 
\begin{equation*}
\big\|(M_Rf)'\big\|_{L^1(\R)} \leq \|f'\|_{L^1(\R)} \ \ \ {  \rm and} \ \ \ \big\|(M_Lf)'\big\|_{L^1(\R)} \leq \|f'\|_{L^1(\R)}.
\end{equation*}
It then follows that $M_Rf$ and $M_Lf$ are in fact absolutely continuous. {\it The reader should keep this fact in mind throughout this section.}
\begin{lemma}[Reduction to one-sided maximal functions]\label{reduction to lateral max} 
Let $\mc{X}$ be the space of functions $f:\R \to \R$ satisfying the following conditions: (i) $f$ is absolutely continuous; (ii) $\lim_{x\to \pm \infty} f(x) = 0$ and (iii) $f' \in L^1(\R)$. Let $h$ and $g$ be two functions in $\mc{X}$ and let $\{h_{j}\}_{j\geq 1}$ and $\{g_{j}\}_{j\geq 1}$ be two sequences in $\mc{X}$ such that $\|h'_{j}-h'\|_{L^{1}(\R)}\to 0$ and $\|g'_{j}-g'\|_{L^{1}(\R)}\to 0$ when $j\to\infty$. Define $f_{j}:=\max\{g_{j},h_{j}\}$ for each $j \in \N$ and $f:=\max\{g,h\}$. Then
$$
\|f'_{j}-f'\|_{L^{1}(\R)}\to0.
$$
when $j\to\infty$.
\end{lemma}

\begin{proof}
Observe first that $f \in \mc{X}$. The fact that $f$ is absolutely continuous follows directly from the definition of absolute continuity, and the fact that $\lim_{x\to \pm} f(x) = 0$ is obvious. Then $f$ is differentiable almost everywhere. Assume that $f$, $g$ and $h$ are differentiable outside a set $A$ of measure zero. Let $B$ be the set of isolated points of the set $\{x \in \R; \ g(x) = h(x)\}$ (note that $B$ has measure zero as well since it is countable). Then 
\begin{equation}\label{Der_sup}
f' = \left\{
\begin{array}{lcc}
h'  & \ {\rm in} \ & \{ g  < h\} \setminus A\\
g' = h'  & \ {\rm in} \ & \{ g = h\} \setminus \{A \cup B\}\\
g'  & \ {\rm in} \ & \{ g >  h\} \setminus A.
\end{array}
\right.
\end{equation}
In particular, this implies condition (iii). Similarly, each $f_j \in \mc{X}$.

\smallskip

Define the sets
$$
X=\{x\in\R;\,g(x) < h(x)\}\ ;\ Y=\{x\in\R;\,g(x)=h(x)\} \ ;\  Z=\{x\in\R; \,g(x) > h(x)\}
$$
and, analogously, define
$$
X_{j}=\{x\in\R;\,g_{j}(x) < h_{j}(x)\}\ ; \ Y_{j}=\{x\in\R;\,g_{j}(x)=h_{j}(x)\}\ ; \ Z_{j}=\{x\in\R; \,g_{j}(x) > h_{j}(x)\}
$$
for each $j$. From \eqref{Der_sup} we find
\begin{align}\label{Lem10_Eq1_div}
\begin{split}
\int_{X}|f'_{j}(y)-f'(y)|\,\dy&= \int_{X}|f'_{j}(y)-h'(y)|\,\dy\\
&=\int_{X\cap (X_{j}\cup Y_{j})}|f'_{j}(y)-h'(y)|\,\dy+\int_{X\cap Z_{j}}|f'_{j}(y)-h'(y)|\,\dy\\
&=\int_{X\cap (X_{j}\cup Y_{j})}|h'_{j}(y)-h'(y)|\,\dy+\int_{X\cap Z_{j}}|g'_{j}(y)-h'(y)|\,\dy\\
&\leq \|h'_{j}-h'\|_{L^{1}(\R)}+\int_{X\cap Z_{j}}|g'_{j}(y)-g'(y)|\,\dy+\int_{X\cap Z_{j}}|g'(y)-h'(y)|\,\dy\\
&\leq \|h'_{j}-h'\|_{L^{1}(\R)}+\|g'_{j}-g'\|_{L^{1}(\R)}+\int_{\R}|g'(y)-h'(y)|\,\chi_{(X\cap Z_{j})}(y)\,\dy\\
&\to 0 \ \ \text{when} \ \ j\to\infty.
\end{split}
\end{align}
The last convergence to $0$ is a consequence of the hypothesis and the dominated convergence theorem. In fact, if $x \in X$, from the fundamental theorem of calculus we have
\begin{align*}
0 < h(x) - g(x) & = \int_{-\infty}^{x} \{ h'(t) - g'(t)\}\,\dt\\
& =  \int_{-\infty}^{x} \{ h'(t) - h_j'(t)\}\,\dt + \int_{-\infty}^{x} \{ h_j'(t) - g_j'(t)\}\,\dt + \int_{-\infty}^{x} \{ g_j'(t) - g'(t)\}\,\dt\\
& \leq \|h'_{j}-h'\|_{L^{1}(\R)} + \big(h_j(x) - g_j(x)\big) + \|g'_{j}-g'\|_{L^{1}(\R)}\,,
\end{align*}
and we must have $h_j(x) > g_j(x)$ (in particular $x \notin Z_j$) for $j$ large. Hence the function $\chi_{(X\cap Z_{j})}(x)$ tends pointwise to zero as $j \to \infty$.

\smallskip

In a similar way, one can show that 
$$
\int_{Y}|f'_{j}(y)-f'(y)|\,\dy\to 0 \ \ \text{and} \ \ \int_{Z}|f'_{j}(y)-f'(y)|\,\dy\to 0
$$
as $j \to \infty$. The lemma follows from this.
\end{proof}

In order to prove Theorem \ref{Thm1} we must show that if $f\in W^{1,1}(\R)$ and $\{f_{j}\}_{j \geq 1}\subset W^{1,1}(\R)$ are such that $\|f_{j}-f\|_{W^{1,1}(\R)}\to0$ as $j\to\infty$, then $ \|(\widetilde Mf_{j})'-(\widetilde Mf)'\|_{L^{1}(\R)}\to0$ as $j\to\infty$. Due to \eqref{Expression_M_MR_ML} and Lemma \ref{reduction to lateral max} it suffices to show that
$$
\big\|(M_{R}f_{j})'-(M_{R}f)'\big\|_{L^{1}(\R)}\to0 \ \ \text{and}\ \  \big\|(M_{L}f_{j})'-(M_{L}f)'\big\|_{L^{1}(\R)}\to0$$
as $j \to \infty$. We prove the first convergence, and the second one is analogous.

\subsection{Preliminaries II - The derivative of the maximal function} For $f \in L^1(\R)$ and absolutely continuous let us define the set of ``good radii'' of $M_R$ at the point $x$ by
$$
\mc{R}f(x)=\left\{r\in[0,\infty):\ M_{R}f(x)=\frac{1}{r}\int_{x}^{x+r}|f(y)|\,\dy   \right\}.
$$
When $r=0$, it is understood that we take the value of $|f(x)|$. The set $\mc{R}f(x)$ is nonempty and  compact (if $f\not\equiv 0$) for all $x\in\R$, since $f \in L^1(\R)$. 

\begin{lemma}\label{puntos de acumulacion}
Let $f\in W^{1,1}(\R)$ and $\{f_{j}\}_{j \geq 1}\subset W^{1,1}(\R)$ be such that $\|f_{j}-f\|_{W^{1,1}(\R)}\to0$ as $j\to\infty$. For a fixed $x$, define a sequence $\{r_j\}_{j \in \N}$ by picking $r_{j}\in \mc{R}f_{j}(x)$; if $r$ is an accumulation point of $\{r_{j}\}_{j \in \N},$ then $r\in \mc{R}f(x)$.
\end{lemma}
\begin{proof}
Let us assume without loss of generality that $r_j \to r$ as $j \to \infty$ (by passing to a subsequence if necessary). Observe initially (the second inequality below is due to the fundamental theorem of calculus) that
\begin{equation}\label{Pf_Lem11_Sob_emb}
\|M_Rf_j - M_Rf\|_{L^{\infty}(\R)} \leq \|f_j - f\|_{L^{\infty}(\R)} \leq \|f_j - f\|_{W^{1,1}(\R)}.
\end{equation}

\noindent {\it Case 1}. $r \neq 0$. Using \eqref{Pf_Lem11_Sob_emb} we have
\begin{align*}
M_{R}f_{j}(x)&=\frac{1}{r_{j}}\int_{x}^{x+r_{j}}|f_{j}(y)|\, \dy\\
&=\frac{1}{r_{j}}\int_{x}^{x+r_{j}}\big(|f_{j}(y)|-|f(y)|\big)\, \dy+
\left(\frac{1}{r_{j}}\int_{x}^{x+r_{j}}|f_{}(y)|\,\dy-\frac{1}{r}\int_{x}^{x+r}|f(y)|\,\dy\right)+\frac{1}{r}\int_{x}^{x+r}|f(y)|\,\dy\\
&\to \frac{1}{r}\int_{x}^{x+r}|f(y)|\,\dy
\end{align*}
as $j \to \infty$. From \eqref{Pf_Lem11_Sob_emb} we also know that $M_{R}f_{j}(x)\to M_{R}f(x)$, hence
$$
M_{R}f(x)=\frac{1}{r}\int_{x}^{x+r}|f(y)|\,\dy.
$$
and $r \in \mc{R}f(x)$.

\smallskip

\noindent {\it Case 2}. $r=0$.

\smallskip

{\it Subcase 2.1}. If $r_{j}=0$ for $j \geq j_0$. Then $M_{R}f_{j}(x)=|f_{j}(x)|\to|f(x)|$ by \eqref{Pf_Lem11_Sob_emb}. Another application of \eqref{Pf_Lem11_Sob_emb} gives us that $r=0\in \mc{R}f(x)$.

\smallskip

{\it Subcase 2.2}. If there is a subsequence of $r_{j}>0$. Then, along this subsequence, 
$$ M_{R}f_{j}(x)=\left(\frac{1}{r_{j}}\int_{x}^{x+r_{j}}\big(|f_{j}(y)|-|f(y)|\big)\,\dy+\frac{1}{r_{j}}\int_{x}^{x+r_{j}}|f(y)|\,\dy\right)
$$
This implies that
$$
-\|f_{j}-f\|_{L^\infty(\R)}+\frac{1}{r_{j}}\int_{x}^{x+r_{j}}|f(y)|\,\dy\leq M_{R}f_{j}(x)\leq \|f_{j}-f\|_{L^\infty(\R)}+\frac{1}{r_{j}}\int_{x}^{x+r_{j}}|f(y)|\,\dy.
$$
It follows from \eqref{Pf_Lem11_Sob_emb} and the Lebesgue differentiation theorem that $M_{R}f_{j}(x)\to |f(x)|$ as $j \to \infty$. Another application of \eqref{Pf_Lem11_Sob_emb} gives us that $r=0\in \mc{R}f(x)$.  
\end{proof}

Our next lemma provides an expression for the derivative of the one-sided maximal function. It follows from a result of Haj\l asz and Mal\'{y} \cite{HM}.

\begin{lemma}[Derivative of the one-sided maximal function] \label{representacion de la derivada}
Let $f\in W^{1,1}(\R)$. Almost every point $x \in \R$ has the following property: for any $r \in \mc{R}f(x)$ we have
\begin{equation}\label{Lem_der_formula}
(M_{R}f)'(x)=\frac{1}{r}\int_{x}^{x+r}|f|'(y)\,\dy.
\end{equation}
\end{lemma}
\noindent Note: When $r=0$, it is understood that we take the value $|f|'(x)$ on the right-hand side of \eqref{Lem_der_formula}.
\begin{proof}
Let us recall that a function $f: \R \to \R$ is said to be {\it approximately differentiable} at a point $x_0\in \R$ if there exists a real number $\alpha$ such that, for any $\varepsilon >0$, the set
$$A_{\varepsilon} = \left\{ x \in \R; \ \frac{|f(x) - f(x_0) - \alpha(x- x_0)|}{|x-x_0|} < \varepsilon \right\}$$
has $x_0$ as a density point. In this case, the number $\alpha$ is called the approximate derivative of $f$ at $x_0$ and it is uniquely determined. Of course, if $f$ is differentiable at $x_0$ then it is approximately differentiable at $x_0$, and the classical and approximate derivatives coincide.

\smallskip

An argument of Haj\l asz and Mal\'{y} \cite[Theorems 1 and 2]{HM} shows that if a function $f \in W^{1,1}(\R)$ then $M_Rf$ is approximately differentiable a.e. and the approximate derivative is given by \eqref{Lem_der_formula} for almost all $x \in \R$ (and in this case for all $r \in \mc{R}f(x)$). Since we also know that $M_Rf$ is absolutely continuous, the fact that it is differentiable almost everywhere in the classical sense concludes the proof.
\end{proof}

\subsection{Preliminaries III - Pointwise convergence in the disconnecting set} The next result is basic yet very useful for our purposes.
\begin{lemma}\label{pass_modulus}
Let $f\in W^{1,1}(\R)$ and $\{f_{j}\}_{j \geq 1}\subset W^{1,1}(\R)$ be such that $\|f_{j}-f\|_{W^{1,1}(\R)}\to0$ as $j\to\infty$. Then $\||f_{j}|-|f|\|_{W^{1,1}(\R)}\to0$ as $j\to\infty$.
\end{lemma}
\begin{proof}
Since $\big||f_{j}|-|f|\big| \leq |f_j -f|$ pointwise, it follows that $\||f_{j}|-|f|\|_{L^{1}(\R)}\to0$ as $j \to \infty$. Noting that $|f| = \max\{f, -f\}$, the fact that $\||f_{j}|'-|f|'\|_{L^{1}(\R)}\to0$ follows directly from Lemma \ref{reduction to lateral max}.
\end{proof}

As a consequence of the previous lemmas we obtain the following pointwise convergence result.

\begin{lemma}\label{convergencia puntual donde discolan}
Let $f\in W^{1,1}(\R)$ and $\{f_{j}\}_{j \geq 1}\subset W^{1,1}(\R)$ be such that $\|f_{j}-f\|_{W^{1,1}(\R)}\to0$ as $j\to\infty$. Then
$$
(M_{R}f_{j})'(x)\to (M_{R}f)'(x) 
$$
for almost all points $x$ in the disconnecting set $D:=\{x \in \R;\, M_{R}f(x)>|f(x)|\}$.
\end{lemma}

\begin{proof} Assume that $D$ has positive measure, otherwise there is nothing to prove (in particular this implies that $f \not\equiv 0$). Let $E_0$ be the set of measure zero for which the statement of Lemma \ref{representacion de la derivada} fails for $f$ (i.e. for all $x \in \R\setminus E_0$, the function $M_{R}f$ is differentiable at $x$, $|f|$ is differentiable at $x$, and \eqref{Lem_der_formula} holds for all $r \in \mc{R}f(x)$). Analogously, let $E_j$ be the set of measure zero for which the statement of Lemma \ref{representacion de la derivada} fails for $f_j$. Let $E = \cup_{j\geq 0} E_j$. Note that $E$ still has measure zero.

\smallskip

Let $x \in D\setminus E$. Then there exist $\delta = \delta(x)>0$ and $N = N(x) <\infty$ such that $\mc{R}f(x)\subset [\delta,N]$. By \eqref{Pf_Lem11_Sob_emb} we have that $M_Rf_j(x) \to M_Rf(x)$ as $j \to \infty$, hence there exists $j_0 = j_0(x)$ such that $M_Rf_j(x) \geq M_Rf(x)/2 >0$ for $j \geq j_0$ (the fact that $M_Rf(x) >0$ follows since $x\in D$). Since $\|f_{j}-f\|_{W^{1,1}(\R)}\to0$ as $j\to\infty$ we find that $\{\|f_j\|_{L^1(\R)}\}_{j\geq1}$ are uniformly bounded. It then follows that there exists a large constant $M$ (it does not hurt to suppose $M>2N$) such that $\mc{R}f_{j}(x) \subset [0,M]$ for $j \geq j_0$.

\smallskip

By Lemma \ref{puntos de acumulacion} there exists $j_1 = j_1(\delta,N,x)$ such that, for every $j\geq j_1$, we have $\mc{R}f_{j}(x) \subset [\frac{\delta}{2},2N]$. In fact, if this were not the case, we could produce a sequence of $r_{j_\ell} \in \mc{R}f_{j_\ell}(x)$, with $j_\ell \to \infty$, trapped in the compact set $[0,\frac{\delta}{2}]\cup[2N,M]$, from which we could extract a convergent subsequence to arrive at a contradiction to Lemma \ref{puntos de acumulacion}. If $r_{j}\in \mc{R}f_{j}(x)$ for $j \geq j_1$, by Lemma  \ref{representacion de la derivada} we obtain
$$
(M_{R}f_{j})'(x)=\frac{1}{r_{j}}\int_{x}^{x+r_{j}}|f_{j}|'(y)\,\dy\leq \frac{1}{r_{j}}\big\||f_{j}|'\big\|_{L^1(\R)}<C<\infty.
$$

We claim that $(M_{R}f)'(x)$ is the unique accumulation point of the sequence \linebreak $\{(M_{R}f_{j})'(x)\}_{j \geq 1}$, which would conclude the proof. Assume that $(M_{R}f_{j_{k}})'(x)$ is convergent through a certain subsequence $\{j_k\}$. By passing to a further subsequence, if necessary, we may assume that $r_{j_{k}}\to r\in \mc{R}f(x)$ (note the use of Lemma \ref{puntos de acumulacion} here). Using Lemmas \ref{representacion de la derivada}  and \ref{pass_modulus} we find
\begin{align*}
(M_{R}f_{j_{k}})'(x)&=\frac{1}{r_{j_{k}}}\int_{x}^{x+r_{j_{k}}}|f_{j_{k}}|'(y)\,\dy\\
&\to\frac{1}{r}\int_{x}^{x+r}|f|'(y)\,\dy=(M_{R}f)'(x)
\end{align*}
as $j_{k}\to\infty$. This verifies our claim.
\end{proof}

\subsection{Proof of Theorem \ref{Thm1}} Let $f\in W^{1,1}(\R)$ and $\{f_{j}\}_{j \geq 1}\subset W^{1,1}(\R)$ be such that $\|f_{j}-f\|_{W^{1,1}(\R)}\to0$ as $j\to\infty$. As already observed, we need to show that $
\|(M_{R}f_{j})'-(M_{R}f)'\|_{L^{1}(\R)}\to0$ as $j \to \infty$. From Lemma \ref{pass_modulus} we may assume without loss of generality that $f_j$ and $f$ are all nonnegative.

\subsubsection{Connecting and disconnecting sets}\label{sec_connect_disconnect}
 We define the connecting and disconnecting sets
\begin{equation}\label{detachment}
C:=\{x\in\R;\ M_{R}f(x)=f(x)\}\ \ \text{and}\ \ D:=\{x\in\R;\ M_{R}f(x)>f(x)\},
\end{equation}
and analogously, for each $j \in \N$, we define
$$
C_{j}:=\{x\in\R;\ M_{R}f_{j}(x)=f_{j}(x)\}\ \ \text{and}\ \ D_{j}:=\{x\in\R;\ M_{R}f_{j}(x)>f_{j}(x)\}.
$$
Naturally, the set $D$, as well as each set $D_j$, is an open set. We now make two observations which are crucial for our argument. The first is due to Tanaka: an application of the argument of \cite[Lemma 2]{Ta} shows that\footnote{Tanaka chooses to present all arguments for the left-sided maximal function $M_Lf$ and hence \cite[Lemma 2(a)]{Ta} states that on the appropriate disconnecting set for $M_L$, $M_Lf$ is non-increasing; adapting this argument shows on the other hand that on the disconnecting set $D$ of the right-sided maximal function, $M_R$ is non-decreasing, as we state.}
\begin{equation}\label{crucial_1}
(M_Rf)'(x) \geq 0 \ \ {\rm for \ a.e} \ \ x \in D.
\end{equation}
Note that for a point $x$ of differentiability of $f$ such that $f'(x) >0$, one must have $x \in D$. This leads us to our second observation,
\begin{equation}\label{crucial_2}
(M_Rf)'(x) = f'(x) \leq 0 \ \ {\rm for \ a.e} \ \ x \in C.
\end{equation}
Similar observations hold for each $f_j$ and the sets $D_j, C_j$.

\smallskip

Now for each $j$ we split our main integral into four nonnegative terms
\begin{align}\label{cuatro integrales}
\begin{split}
\int_{\R}\big|(M_{R}f_{j})'(x)& -(M_{R}f)'(x)\big|\,\dx\\
& =\int_{C\cap C_{j}}|f_{j}'(x)-f'(x)|\,\dx+\int_{D\cap C_{j}}\big|f_{j}'(x)-(M_{R}f)'(x)\big|\,\dx\\
&\ \ \ \ \ \ \ +\int_{C\cap D_{j}}\big|(M_{R}f_{j})'(x)-f'(x)\big|\,\dx+\int_{D\cap D_{j}}\big|(M_{R}f_{j})'(x)-(M_{R}f)'(x)\big|\,\dx\\
&:=(I)_{j}+(II)_{j}+(III)_{j}+(IV)_{j}\,,
\end{split}
\end{align}
and proceed with the analysis of each of them.

\subsubsection{Analysis of $(I)_j$ and $(II)_j$} Let us adopt the usual notation $o(1)$ for a quantity whose absolute value tends to $0$ as $j \to \infty$. The first term is the easiest to deal with since
$$
(I)_{j}\leq \|f_{j}'-f'\|_{L^{1}(\R)} = o(1).
$$
As for the second one, note that
\begin{align}\label{Pf_Thm1_Eq_II_1}
\begin{split}
\int_{D\cap C_{j}}|f_{j}'(x)|\,\dx&\leq \int_{D_{}\cap C_{j}}|f'(x)|\,\dx+\|f_{j}'-f'\|_{L^{1}(\R)}\\
&=\int_{\R}|f'(x)|\,\chi_{D\cap C_{j}}(x)\,\dx +\|f_{j}'-f'\|_{L^{1}(\R)}
\end{split}
\end{align}
and
\begin{align}\label{Pf_Thm1_Eq_II_2}
\int_{D\cap C_{j}}\big|(M_{R}f)'(x)\big|=\int_{\R}\big|(M_{R}f)'(x)\big|\,\chi_{D\cap C_{j}}(x)\,\dx.
\end{align}
From \eqref{Pf_Lem11_Sob_emb} we know that $\chi_{D\cap C_{j}}(x)\to 0$ pointwise everywhere as $j \to \infty$. Hence, from \eqref{Pf_Thm1_Eq_II_1}, \eqref{Pf_Thm1_Eq_II_2} and the dominated convergence theorem, we obtain
\begin{equation}\label{Conclusion_II}
(II)_{j} = o(1).
\end{equation}

\subsubsection{The Brezis-Lieb reduction} From Lemma \ref{convergencia puntual donde discolan} and the classical Brezis-Lieb lemma \cite{BL} we know that 
$$(II)_j + (IV)_j = \int_{D}\big|(M_{R}f_{j})'(x) -(M_{R}f)'(x)\big|\,\dx \to 0$$
as $j \to \infty$ if and only if
\begin{equation}\label{BLReduction}
\int_{D}\big|(M_{R}f_{j})'(x)\big|\,\dx \to \int_{D}\big|(M_{R}f)'(x)\big|\,\dx
\end{equation}
as $j \to \infty$.

\smallskip

From \eqref{crucial_1} and the fundamental theorem of calculus (applied to each connected component of $D_j$ and summing up over this countable number of open intervals) we observe that for each $j$, 
\begin{equation}\label{Fund_Thm_Calc}
\int_{D_{j}}\big|(M_{R}f_{j})'(x)\big|\,\dx=\int_{D_{j}}(M_{R}f_{j})'(x)\,\dx = \int_{D_{j}}f_{j}'(x)\,\dx.
\end{equation}
As a result of \eqref{Fund_Thm_Calc}, writing $D_j$ as the disjoint union $(C \cap D_j) \cup (D \cap D_j)$, we see that for each $j \in \N$ we have the following dichotomy: either 
\begin{equation}\label{Case1_dichotomy} 
\int_{C \cap D_{j}}\big|(M_{R}f_{j})'(x)\big|\,\dx\leq\int_{C \cap D_{j}}f_{j}'(x)\,\dx
\end{equation}
or
\begin{equation}\label{Case2_dichotomy}
\int_{D \cap D_{j}}\big|(M_{R}f_{j})'(x)\big|\dx\leq\int_{D \cap D_{j}}f_{j}'(x)\,\dx.
\end{equation}

\subsubsection{Dichotomy: Case 1} Assume that we go over the subsequence of $j$'s such that \eqref{Case1_dichotomy} holds. Then, using \eqref{crucial_1}, \eqref{crucial_2} and \eqref{Case1_dichotomy}, we have
\begin{align*}
(III)_{j}& =\int_{C\cap D_{j}}\big|(M_{R}f_{j})'(x)-f'(x)\big|\,\dx\\
&=\int_{C\cap D_{j}}\big((M_{R}f_{j})'(x)-f'(x)\big)\,\dx\\
&\leq \int_{C\cap D_{j}}\big(f_{j}'(x)-f'(x)\big)\,\dx\\
&\leq \|f_{j}'-f'\|_{L^{1}(\R)}.
\end{align*}
Hence
\begin{align}\label{case1 3j}
(III)_{j}=o(1)
\end{align}
along this subsequence. Applying \eqref{Fund_Thm_Calc} to $f_j$ in the fourth identity below and to $f$ in the sixth, we find that for any fixed $j$,
\begin{align}\label{Mega_comput}
\begin{split}
\int_{D}\big|&(M_{R}f_{j})' (x)\big|\,\dx  = \int_{D\cap D_j}\big|(M_{R}f_{j})'(x)\big|\,\dx +  \int_{D\cap C_j}\big|(M_{R}f_{j})'(x)\big|\,\dx + (III)_j - (III)_j\\
& = \int_{D\cap D_j}(M_{R}f_{j})'(x)\,\dx +  \int_{D\cap C_j}\big|f_{j}'(x)\big|\,\dx + \int_{C\cap D_{j}}\big((M_{R}f_{j})'(x)-f'(x)\big)\,\dx - (III)_j\\
& = \int_{D_j}(M_{R}f_{j})'(x)\,\dx - \int_{C\cap D_{j}}f'(x)\,\dx + \int_{D\cap C_j}\big|f_{j}'(x)\big|\,\dx - (III)_j\\
& = \int_{D_j} f_{j}'(x)\,\dx  - \int_{C\cap D_{j}}f'(x)\,\dx + \int_{D\cap C_j}\big|f_{j}'(x)\big|\,\dx - (III)_j\\
& = \int_{D_j} \big(f_{j}'(x) - f'(x) \big)\,\dx  + \int_{D}f'(x)\,\dx  -  \int_{D\cap C_{j}}f'(x)\,\dx + \int_{D\cap C_j}\big|f_{j}'(x)\big|\,\dx - (III)_j\\
& = \int_{D_j} \big(f_{j}'(x) - f'(x) \big)\,\dx  + \int_{D}\big|(M_Rf)'(x)\big|\,\dx  -  \int_{D\cap C_{j}}f'(x)\,\dx + \int_{D\cap C_j}\big|f_{j}'(x)\big|\,\dx - (III)_j.
\end{split}
\end{align}
(In fact, note that this identity was derived independently of \eqref{Case1_dichotomy} or \eqref{Case2_dichotomy} and hence holds regardless of whether we are in Case 1 or Case 2 of our dichotomy). Using \eqref{Pf_Thm1_Eq_II_1}, \eqref{case1 3j} and the fact that $\chi_{D\cap C_{j}}(x)\to 0$ pointwise, the dominated convergence theorem gives us
\begin{equation*}
\int_{D}\big|(M_{R}f_{j})'(x)\big|\,\dx  \to \int_{D}\big|(M_Rf)'(x)\big|\,\dx 
\end{equation*}
as $j \to \infty$ (along this subsequence of $j$'s of this Case 1), and from the Brezis-Lieb reduction \eqref{BLReduction} we conclude that
$$(II)_j + (IV)_j  = o(1).$$
In light of \eqref{Conclusion_II} and the non-negativity of each term, this implies that 
$$(IV)_j  = o(1)$$
and we are done for this subsequence of $j$'s.

\subsubsection{Dichotomy: Case 2} Assume now that we go over the subsequence of $j$'s such that \eqref{Case2_dichotomy} holds. From Lemma \ref{convergencia puntual donde discolan} and Fatou's lemma we get
\begin{equation}\label{Fat_Lat_Lim_1}
\int_{D}\big|(M_{R}f)'(x)\big|\,\dx\leq\liminf_{j\to\infty}\int_{D}\big|(M_{R}f_{j})'(x)\big|\,\dx.
\end{equation}
Using \eqref{Case2_dichotomy} and \eqref{crucial_2} respectively, in the second line below, and \eqref{Fund_Thm_Calc} (applied to $f$) in the fourth line below, we obtain for each $j$,
\begin{align*}
\int_{D}\big|(M_{R}f_{j})'(x)\big|\,\dx&=\int_{D\cap D_{j}}\big|(M_{R}f_{j})'(x)\big|\,\dx+\int_{D \cap C_{j}}\big|(M_{R}f_{j})'(x)\big|\,\dx\\
&\leq \int_{D\cap D_{j}}f_{j}'(x)\,\dx-\int_{D \cap C_{j}}f_{j}'(x)\,\dx\\
&= \int_{D}f'(x) + \int_{D \cap D_{j}}\big(f_{j}'(x)-f'(x)\big)\,\dx\ -\int_{D \cap C_{j}} f_{j}'(x) \,\dx - \int_{D \cap C_{j}} f'(x)\,\dx\\
& = \int_{D}\big|(M_{R}f)'(x)\big|\,\dx + \int_{D \cap D_{j}}\big(f_{j}'(x)-f'(x)\big)\,\dx\ -\int_{D \cap C_{j}} f_{j}'(x) \,\dx - \int_{D \cap C_{j}} f'(x)\,\dx.
\end{align*}
The first term we preserve as the main term; the second term is dominated by $\|f_j' - f'\|_{L^1(\R)}$ which is $o(1)$; the third term and fourth terms are $o(1)$ by \eqref{Pf_Thm1_Eq_II_1} and dominated convergence since $\chi_{D\cap C_{j}}(x)\to 0$ pointwise. It then follows from the previous computation that  (along the subsequence of $j$'s we consider)
\begin{equation}\label{Fat_Lat_Lim_2}
\limsup_{j\to\infty}\int_{D}\big|(M_{R}f_{j})'(x)\big|\,\dx \leq \int_{D}\big|(M_{R}f)'(x)\big|\,\dx
\end{equation}
From \eqref{Fat_Lat_Lim_1} and \eqref{Fat_Lat_Lim_2} we obtain 
\begin{equation}\label{BL_Conc_case_2_dic}
\int_{D}\big|(M_{R}f_{j})'(x)\big|\,\dx  \to \int_{D}\big|(M_Rf)'(x)\big|\,\dx 
\end{equation}
as $j \to \infty$ (along this subsequence of $j$'s), and from the Brezis-Lieb reduction \eqref{BLReduction} we conclude that 
$$(II)_j + (IV)_j  = o(1).$$
It follows from \eqref{Conclusion_II} that
$$(IV)_j  = o(1).$$
Returning to identity \eqref{Mega_comput} (which we recall holds for any fixed $j$) we use \eqref{BL_Conc_case_2_dic}, \eqref{Pf_Thm1_Eq_II_1} and the dominated convergence theorem (since $\chi_{D\cap C_{j}}(x)\to 0$ pointwise), to finally arrive at
$$(III)_j  = o(1)$$
along this subsequence of $j's$. This concludes the proof.

\section{Acknowledgments}
\noindent E.C. acknowledges support from CNPq-Brazil grants $305612/2014-0$ and $477218/2013-0$, and FAPERJ grant $E-26/103.010/2012$. L. B. P. has been supported in part by NSF DMS-1402121 and NSF CAREER grant DMS-1652173. E. C. and L. B. P. thank the University of Oxford ``Developing Leaders'' scheme for support for a research visit during this collaboration. J. M. acknowledges support from CNPq-Brazil. The authors are thankful to Kevin Hughes for helpful discussions during the preparation of this manuscript.

\end{document}